\documentclass[reqno]{amsart}
\usepackage{amsfonts}
\usepackage{amsmath}
\usepackage{amssymb}

\newtheorem{theorem}{Theorem}

\newtheorem{definition}[theorem]{Definition}
\newtheorem{example}[theorem]{Example}

\newtheorem{lemma}[theorem]{Lemma}

\newtheorem{proposition}[theorem]{Proposition}
\newtheorem{remark}[theorem]{Remark}

\input{tcilatex}

\begin{document}
\title[Nonlinear Spaces]{Certain Results For a Class of Nonlinear Functional
Spaces}
\author{Kamal Soltanov}
\address{Kamal Soltanov: Faculty of Science and Literature, Department of
Mathematics, I\u{g}d\i r University, 76000, I\u{g}d\i r, Turkey.}
\email{kamal.soltanov@igdir.edu.tr}
\author{U\u{g}ur Sert}
\address{U\u{g}ur Sert: Faculty of Science, Department of Mathematics,
Hacettepe University, 06800, Beytepe, Ankara, Turkey.}
\email{usert@hacettepe.edu.tr}
\subjclass[2010]{Primary 46A99, 46E30, 46E35, 46T99; Secondary 26D20, 26D99,
35D30, 35J62, 35K61.}
\keywords{pn-spaces, integral inequalities, nonlinear differential
equations, embedding theorems.}

\begin{abstract}
In this article, we study properties of a class of functional spaces which
arise from investigation of nonlinear differential equations. We establish
some integral inequalities then by applying these inequalities, we prove
some lemmas and theorems which indicate the relation of these spaces
(pn-spaces) with the Lebesgue and Sobolev spaces in the case when pn-spaces
with constant and variable exponents
\end{abstract}

\maketitle

\section{Introduction}

This paper is concerned with some features of a class of functional spaces
which are emerged from investigation of nonlinear differential equations.
Studying boundary value problems require to examine and understand the
functional spaces which are directly related with the considered problem. In
other words, it is required to work on the domain of the operator generated
by the addressed boundary value problem. We specify that it is better to
study each BVPs on its own space. Furthermore, detailed analysis of these
spaces and examining their topology, structure etc. cause to gain better
results of the possed problem (for example regularity of the solution).

The spaces generated by boundary value problems for the linear differential
equations are generally linear spaces such as Sobolev spaces and different
generalizations of them. Apart from boundary value problems for linear
differential equations, the spaces generated by nonlinear differential
equations (essentially the domain of the corresponding operator) are subsets
of linear spaces and do not have linear structure. The class of spaces of
this type were introduced and investigated by Soltanov in the abstract case
(see, e. g. \cite{S1}-\cite{S6}), and also in the case of functions spaces
(see, e. g. \cite{S3}-\cite{S10} and references therein where various
subsets of linear spaces of this type were searched). In the mentioned
articles, topology of these spaces were investigated and shown that under
what circumstances they are metric or pseudo-metric spaces. Starting from
these features of the introduced spaces, they were defined as the class of
pseudo-normed spaces or pn-spaces and the class of quasi-pseudo normed
spaces or $qn$-spaces.

In this work, we focus on the characteristics of certain class of functional
pn-spaces. Essentially, we deal with the following class of functional
pn-spaces. \newline
Let $\Omega \subset \mathbb{R}^{n}\left( n\geq 1\right) $ be bounded domain
with sufficiently smooth boundary. Here the class of functions $u:\Omega
\longrightarrow 
\mathbb{R}
$ of the following type will be investigated 
\begin{equation}
S_{m,\alpha ,\beta }\left( \Omega \right) :=\left\{ u\in L^{1}\left( \Omega
\right) \mid \left[ u\right] _{S_{m,\alpha ,\beta }\left( \Omega \right)
}^{\alpha +\beta }<\infty \right\}   \tag{1.1}  \label{01}
\end{equation}%
where $\alpha \geq 0,$ $\beta \geq 1$ are real numbers and $m$ is an integer
and 
\begin{equation*}
\left[ u\right] _{S_{m,\alpha ,\beta }\left( \Omega \right) }^{\alpha +\beta
}:=\sum_{0\leq \left\vert k\right\vert \leq m}\left( \int\limits_{\Omega
}\left\vert u\right\vert ^{\alpha }\left\vert D^{k}u\right\vert ^{\beta
}dx\right) ,\quad D=\left( D_{1},D_{2},...,D_{n}\right) ,
\end{equation*}%
$D_{i}=\frac{\partial }{\partial x_{i}}$, $D^{k}\equiv
D_{1}^{k_{1}}D_{2}^{k_{2}},...,D_{n}^{k_{n}}$, $i=\overline{1,n}$, $%
\left\vert k\right\vert =\sum\limits_{i=1}^{n}k_{i}$. Here, we only address
the cases $m=1,2$.

It is important to note that the following subset of $L^{p}\left( \Omega
\right) $, $p\geq 2$ 
\begin{equation*}
M:=\left\{ u\in L^{1}\left( \Omega \right) \left\vert \ \overset{n}{%
\sum_{i=1}}\left( \int\limits_{\Omega }\left\vert u\right\vert
^{p-2}\left\vert D_{i}u\right\vert ^{2}dx\right) <\infty ,u\left\vert \
_{\partial \Omega }\right. =0\right. \right\} 
\end{equation*}%
was arose in the article of Dubinskii earlier (\cite{D1}, \cite{L}, \cite{D2}%
) while studying the following nonlinear problem: 
\begin{equation}
\frac{\partial u}{\partial t}-\sum_{i=1}^{n}D_{i}\left( \left\vert
u\right\vert ^{p-2}D_{i}u\right) =h(x,t),\quad \left( t,x\right) \in \left(
0,T\right) \times \Omega ,  \tag{1.2}  \label{0.2}
\end{equation}%
\begin{equation*}
u\left( 0,x\right) =u_{0}\left( x\right) ,\quad u\left\vert \ _{\left( 0,T
\right] \times \partial \Omega }=0\right. .
\end{equation*}%
Here, compact inclusion of subset $M$ to the space $L^{p}\left( \Omega
\right) $ and also necessary compactness theorem for analysis of the
mentioned parabolic problem were proved. Later on, different new subsets of $%
L^{1}\left( \Omega \right) $ appeared in the articles of Soltanov (see, e.
g. \cite{S3}, \cite{S4}, \cite{S5}) while studying the mixed problem for the
following nonlinear equation which is type of the Prandtl-von-Mises equation 
\begin{equation}
\frac{\partial u}{\partial t}-\left\vert u\right\vert ^{\rho }\frac{\partial
^{2}u}{\partial x^{2}}=h\left( t,x\right) ,\ \rho >0,\quad \left( t,x\right)
\in \left( 0,T\right) \times \Omega   \tag{1.3}  \label{0.3}
\end{equation}%
For example, one of the emerged class in the case of $\Omega =\left(
a,b\right) \subset \mathbb{R}$ can be expressed in the form 
\begin{equation*}
\left\{ u\in L^{1}\left( \Omega \right) \left\vert \ \int\limits_{\Omega
}\left\vert u\right\vert ^{\alpha }\left\vert D^{2}u\right\vert ^{\beta
}dx<\infty ,u\left( a\right) =u\left( b\right) =0\right. \right\} ,
\end{equation*}%
and also as type of subsets in the form $S_{m,\alpha ,\beta }\left( \Omega
\right) $. Here, we specify that different problems for the equation (\ref%
{0.3}) were studied under various additional conditions as well (see, e. g. 
\cite{O}, \cite{Sa}, \cite{L-P}, \cite{Ts-I}, \cite{Wa}, \cite{Wi}).

Accordingly, in the papers (\cite{S4}, \cite{S5}, etc.) different classes of
sets of this type were examined and it was shown that these sets are
nonlinear topological spaces, moreover they are either metric or
pseudo-metric spaces. Many other properties of the introduced spaces were
investigated as well in these works. For instance, relations of these spaces
amongst themselves and with well known functional spaces (e. g. Lebesgue or
Sobolev spaces etc).

Consequently, in the mentioned works pn-spaces and qn-spaces were defined
with taking into account the principal attributes of the presented spaces.

These spaces may arise from the research of the existence of smooth solution
of the following differential equation 
\begin{equation*}
-\Delta u+u+\left\vert u\right\vert ^{p}u=h\left( x\right) ,\text{ }x\in
\Omega \subset \mathbb{R}^{n},\text{ }n\geq 2
\end{equation*}%
\begin{equation*}
\left( \frac{\partial u}{\partial \eta }+\left\vert u\right\vert ^{\mu
}u\right) \left\vert \ _{\partial \Omega }\right. =\psi \left( x^{\prime
}\right) ,\text{ }x^{\prime }\in \partial \Omega ,\text{ }p,\mu \geq 0
\end{equation*}%
which was studied by Soltanov (\cite{S12}). We emphasize that equation of
this form was considered by many authors who tried to answer various
questions of different problems for this equation, (see, e.g. Berestycki ve
Nirenberg \cite{Be-N}, Brezis \cite{Br}, etc.). In \cite{P}, Pohozaev
employed another approach for this problem that led to gaining distinct
results other than \cite{S12}.

This kind of nonlinear spaces are generated by the differential equations
which ensue from the mathematical models of some processes in flood
mechanics. For an example, we may present the nonlinear equation of type 
\begin{equation*}
\frac{\partial u}{\partial t}-\left\vert u\right\vert ^{p-2}\Delta u=h\left(
x,t\right) ,\text{ }p\geq 2
\end{equation*}%
in where this equation were studied \cite{S4, S11} and \cite{S-A}. Similar
equations were handled by Oleynik \cite{O}, Walter \cite{Wa} only using the
approximation way and Tsutsumi, Ishiwata (\cite{Ts-I}) focused on
understanding the behavior of the solution.

In recent years, there have been an increasing interest in the study of
equations with variable exponents of nonlinearities. The interest in the
study of differential equations that involves variable exponents is
motivated by their applications to the theory of elasticity and
hydrodynamics, in particular the models of electrorheological fluids \cite%
{RUZ} in which a substantial part of viscous energy, the thermistor problem 
\cite{Z1}, image processing \cite{CLR} and modeling of non-Newtonian fluids
with thermo-convective effects \cite{AR} etc.

In the most of these papers that concern with equations which have non
standard growth, authors studied the problems which involve $p(.)$-Laplacian
type equation or equations which fulfill monotonicity conditions where
enable to apply monotonicity methods. Unlike these works, in the articles 
\cite{S-S-1,S-S-2} by investigating some properties of nonlinear spaces with
variable exponent, we developed an approach based on the spaces
corresponding to problem under consideration. It is necessary to note that
the questions mentioned above may arise for the problems which have variable
exponent nonlinearity. Eventually, here we also study variable exponent
nonlinear spaces that are essential for the investigation of the following
type of equations: 
\begin{equation*}
\nabla \cdot \left[ \left( \left\vert \nabla u\right\vert ^{p_{0}\left(
x\right) -2}+\left\vert u\right\vert ^{p_{1}\left( x\right) -2}\right)
\nabla u\right] =h\left( x,u\right) .
\end{equation*}

Since we want to establish the regularity of solution of the nonlinear
differential equations related with mentioned pn-spaces, thus our aim is to
understand the structure and nature of these spaces better that allows to
investigate the characteristics of solutions. For this reason, in this
article we prove some embedding results which indicate the relation of these
spaces between Sobolev and Lebesgue spaces. We show that these spaces are
not merely subsets of Lebesgue spaces also subsets of Sobolev spaces.

This paper is organized as follows: In the next section, we give the
definitions of pn-spaces with variable and constant exponents as well as
recall some basic results for these spaces and variable exponent spaces. In
Section 3, we prove embedding theorems for constant exponent pn-spaces and
give certain results with examples in one dimensional case. In Section 4
firstly, we establish some integral inequalities with variable exponents
which are required to prove embedding theorems of variable exponent
nonlinear spaces then investigate some attributes of variable exponent
pn-spaces.

\section{Preliminaries}

In this section, first we remind certain integral inequalities and facts
about the functional pn-spaces with constant exponent that are concerned in
this paper (for general case see \cite{S1} - \cite{S5} and for functional
case \cite{S1}, \cite{S5}, \cite{S7} etc).

Let $\Omega \subset \mathbb{R}^{n}\left( n\geq 1\right) $ be a bounded
domain with Lipschitz boundary $\partial \Omega.$ (Throughout the paper, we
denote by $\left\vert \Omega \right\vert $ the Lebesgue measure of $\Omega $
).

\begin{lemma}
\label{sbt1} Let $\alpha \geq 0,$ $\beta \geq 1$, $\left\vert \Omega
\right\vert<\infty $ and $i=\overline{1,n}$, then for all $u\in C(\bar{\Omega%
})\cap C^{1}(\Omega )$ the inequality 
\begin{equation}
\int\limits_{\Omega }\left\vert u\right\vert ^{\alpha +\beta }dx\leq
C_{1}\int\limits_{\Omega }\left\vert u\right\vert ^{\alpha }\left\vert
D_{i}u\right\vert ^{\beta }dx+C_{2}\int\limits_{\partial \Omega }\left\vert
u\right\vert ^{\alpha +\beta }dx^{\prime}  \tag{2.1}
\end{equation}
is satisfied. Here, $C_{1}=C_{1}\left( \alpha ,\beta ,\left\vert \Omega
\right\vert\right), C_{2}=C_{2}\left( \left\vert \Omega \right\vert\right)
>0 $ are constants.
\end{lemma}

\begin{lemma}
\label{sbt2} Assume that $\alpha, \alpha _{1}\geq 0$, $\beta \geq 1$ and $%
\beta >\beta _{1}>0,$ $\frac{\alpha _{1}}{\beta _{1}}\geq \frac{\alpha }{%
\beta },$ $\alpha _{1}+\beta _{1}\leq \alpha +\beta $ be satisfied. Then for 
$u\in C(\bar{\Omega})\cap C^{1}(\Omega )$ 
\begin{equation}
\int\limits_{\Omega }\left\vert u\right\vert ^{\alpha _{1}}\left\vert
D_{i}u\right\vert ^{\beta _{1}}dx\leq C_{3}\int\limits_{\Omega }\left\vert
u\right\vert ^{\alpha }\left\vert D_{i}u\right\vert ^{\beta
}dx+C_{4}\int\limits_{\partial \Omega }\left\vert u\right\vert ^{\alpha
+\beta }dx^{\prime }+C_{5}  \tag{2.2}
\end{equation}%
holds. Here, for $r=3, 4, 5$, $C_{r}=C_{r}\left( \alpha, \beta, \alpha _{1},
\beta _{1}, \left\vert \Omega \right\vert\right)>0$ are constants.
\end{lemma}

\begin{lemma}
\label{sbt3} Let $\alpha \geq 0,$ $\beta _{0}+\beta _{1}\geq 2$ and $\beta
_{1}\geq \beta _{0}\geq 0$ be fulfilled. Then for all $u\in C^{1}\left(\bar{%
\Omega}\right)\cap C^{2}\left( \Omega \right)$ 
\begin{align*}
\int\limits_{\Omega }\left\vert u\right\vert ^{\alpha }\left\vert
D_{i}u\right\vert ^{\beta _{0}+\beta _{1}}dx&\leq C_{6}\int\limits_{\Omega
}\left\vert u\right\vert ^{\alpha +\beta _{0}}\left\vert
D_{i}^{2}u\right\vert ^{\beta _{1}}dx \\
&+C_{7}\int\limits_{\partial \Omega }(\left\vert u\right\vert ^{\alpha
+\beta _{0}+\beta _{1}}+\left\vert u\right\vert ^{\alpha +1}\left\vert
D_{i}u\right\vert ^{\beta _{0}+\beta _{1}-1})dx^{\prime}  \tag{2.3}
\end{align*}
holds. Here, for $j=6, 7$, $C_{j}=C_{j}\left( \alpha, \beta, \beta
_{0}\right)>0$ are constants.
\end{lemma}

\begin{definition}
\label{tansbt} Let $\alpha \geq 0,$ $\beta \geq 1$, $\mathbf{k}=\left(
k_{1},...,k_{n}\right) $ is multi-index and $\left\vert \mathbf{k}%
\right\vert =\sum\limits_{i=1}^{n}k_{i}$, $m\in \mathbb{Z}^{+},$ $\Omega
\subset \mathbb{R}^{n}\left( n\geq 1\right) $ is bounded domain with
sufficiently smooth boundary (at least Lipschitz boundary) 
\begin{equation*}
S_{m,\alpha ,\beta }\left( \Omega \right) := \left\{ u\in L^{1}\left( \Omega
\right) \left\vert \ \left[ u\right] _{S_{m,\alpha ,\beta }\left( \Omega
\right) }^{\alpha +\beta }\equiv \sum_{0\leq \left\vert \mathbf{k}%
\right\vert \leq m}\left( \int\limits_{\Omega }\left\vert u\right\vert
^{\alpha }\left\vert D^{\mathbf{k}}u\right\vert ^{\beta }dx\right) <\infty
\right. \right\}
\end{equation*}%
and 
\begin{equation*}
\mathring{S}_{m,\alpha ,\beta }\left( \Omega \right) :=S_{m,\alpha ,\beta
}\left( \Omega \right) \cap \left\{ D^{\mathbf{k}}u\left\vert \ _{\partial
\Omega }\right. \equiv 0,\ 0\leq \left\vert \mathbf{k}\right\vert \leq
m_{0}<m\right\} .
\end{equation*}
\end{definition}

We state a proposition which can be easily proved by the help of Lemma \ref%
{sbt1}-Lemma \ref{sbt3} and Definition \ref{tansbt}.

\begin{proposition}
Assume that $\alpha \geq 0,$ $\beta \geq 1$ then we have the following
equivalence; 
\begin{equation*}
\mathring{S}_{1,\alpha ,\beta }\left( \Omega \right):= \left\{ u\in
L^{1}\left( \Omega \right) \mid \left[ u\right] _{S_{1,\alpha ,\beta }\left(
\Omega \right) }^{\alpha +\beta }\equiv \sum_{i=1}^{n}\left(
\int\limits_{\Omega }\left\vert u\right\vert ^{\alpha }\left\vert
D_{i}u\right\vert ^{\beta }dx\right) <\infty \right\}
\end{equation*}
and \footnote{$S_{1,\alpha ,\beta }\left( \Omega \right) $ is a complete
metric space with the following metric: $\forall u,v\in S_{1,\alpha ,\beta
}\left( \Omega \right) $%
\par
\begin{equation*}
d_{S_{1,\alpha ,\beta }}\left( u,v\right) =\left\Vert \left\vert
u\right\vert ^{\frac{\alpha }{\beta }}u-\left\vert v\right\vert ^{\frac{%
\alpha }{\beta }}v\right\Vert _{W^{1,\beta }\left( \Omega \right) }
\end{equation*}%
} 
\begin{equation*}
\mathring{S}_{2,\alpha ,\beta }\left( \Omega \right):= \left\{ u\in
L^{1}\left( \Omega \right) \mid \left[ u\right] _{S_{2,\alpha ,\beta }\left(
\Omega \right) }^{\alpha +\beta }\equiv \sum_{i=1}^{n}\left(
\int\limits_{\Omega }\left\vert u\right\vert ^{\alpha }\left\vert
D_{i}^{2}u\right\vert^{\beta }dx\right) <\infty \right\}
\end{equation*}
\end{proposition}

\begin{theorem}
\label{sbthom} Let $\alpha \geq 0,$ $\beta \geq 1$ then $g:\mathbb{R}%
\longrightarrow \mathbb{R},$ $g(t):= \left\vert t\right\vert ^{\frac{\alpha 
}{\beta }}t$ is an one to one correspondence from $S_{1,\alpha ,\beta
}(\Omega )$ onto $W^{1,\beta}(\Omega )$.
\end{theorem}

\medskip Now, we recall some basic definitions and results about variable
exponent Lebesgue and Sobolev spaces \cite{ADA,DIE2,FAN1,KOV,MUS}.

Let $\Omega $ be a Lebesgue measurable subset of $\mathbb{R}^{n}$ such that $%
\left\vert \Omega \right\vert >0$. The function set $M\left( \Omega \right) $
denotes the family of all measurable functions $p:\Omega \longrightarrow %
\left[ 1,\infty \right] $ and the set $M_{0}\left( \Omega \right) $ is
defined as, 
\begin{equation*}
M_{0}\left( \Omega \right) :=\left\{ p\in M\left( \Omega \right) :\ 1\leq
p^{-}\leq p\left( x\right) \leq p^{+}<\infty ,\text{ a.e. }x\in \Omega
\right\}
\end{equation*}
where $p^{-}:=\underset{\Omega }{ess}\inf \left\vert p\left( x\right)
\right\vert,\text{ } p^{+}:=\underset{\Omega }{ess}\sup \left\vert p\left(
x\right) \right\vert$.

For $p\in $ $M\left( \Omega \right) ,$ $\Omega _{\infty }^{p}\equiv \Omega
_{\infty }\equiv \left\{ x\in \Omega |\text{ }p\left( x\right) =\infty
\right\}.$ On the set of all functions on $\Omega,$ define the functional $%
\sigma _{p}$ and $\left\Vert .\right\Vert _{p}$ by%
\begin{equation*}
\sigma _{p}\left( u\right) \equiv \int\limits_{\Omega \backslash \Omega
_{\infty }}\left\vert u\right\vert ^{p\left( x\right) }dx+\underset{\Omega
_{\infty }}{ess}\sup \left\vert u\left( x\right) \right\vert
\end{equation*}%
and%
\begin{equation*}
\left\Vert u\right\Vert _{L^{p\left( x\right) }\left( \Omega \right) }\equiv
\inf \left\{ \lambda >0:\text{ }\sigma _{p}\left( \frac{u}{\lambda }\right)
\leq 1\right\} .
\end{equation*}%
If $p\in L^{\infty }\left( \Omega \right) $ then $p\in M_{0}\left( \Omega
\right) $, $\sigma _{p}\left( u\right) \equiv \int\limits_{\Omega
}\left\vert u\right\vert ^{p\left( x\right) }dx$ and the variable exponent
Lebesgue space is defined as follows: 
\begin{equation*}
L^{p\left( x\right) }\left( \Omega \right) := \left\{ u: u\text{ } \text{is
a measurable real-valued function such that}\text{ } \sigma _{p}\left(
u\right) <\infty \right\} .
\end{equation*}%
If $p^{-}>1,$ then the space $L^{p\left( x\right) }\left( \Omega \right) $
becomes a reflexive and separable Banach space with the norm $\left\Vert
.\right\Vert _{L^{p\left( x\right) }\left( \Omega \right)}$ which is
so-called Luxemburg norm.

If $0<\left\vert \Omega \right\vert <\infty ,$ and $p_{1},$ $p_{2}\in
M\left( \Omega \right) $ then the continuous embedding $L^{p_{1}\left(
x\right) }\left( \Omega \right) \subset L^{p_{2}\left( x\right) }\left(
\Omega \right) $ exists $\iff $ $p_{2}\left( x\right) \leq p_{1}\left(
x\right) $ for a.e. $x\in \Omega .$

For $u\in L^{p\left( x\right) }\left( \Omega \right) $ and $v\in L^{q\left(
x\right) }\left( \Omega \right) $ where $p,$ $q\in $ $M_{0}\left( \Omega
\right) $ and $\frac{1}{p\left( x\right) }+\frac{1}{q\left( x\right) }=1$
the following inequalities be satisfied:

\begin{equation*}  \label{holder}
\int\limits_{\Omega }\left\vert uv\right\vert dx\leq 2\left\Vert
u\right\Vert _{L^{p\left( x\right) }\left( \Omega \right) }\left\Vert
v\right\Vert _{L^{q\left( x\right) }\left( \Omega \right), }  \tag{2.4}
\end{equation*}
and 
\begin{equation*}  \label{minmax}
\min \{\left\Vert u\right\Vert _{L^{p\left( x\right) }\left( \Omega \right)
}^{p^{-}}, \left\Vert u\right\Vert _{L^{p\left( x\right) }\left( \Omega
\right) }^{p^{+}}\}\leq \sigma _{p}\left( u\right) \leq \max \{\left\Vert
u\right\Vert _{L^{p\left( x\right) }\left( \Omega \right) }^{p^{-}},
\left\Vert u\right\Vert _{L^{p\left( x\right) }\left( \Omega \right)
}^{p^{+}}\}.  \tag{2.5}
\end{equation*}

\begin{lemma}
\label{denknorm} Let $u$, $u_{k}\in L^{p\left( x\right) }\left( \Omega
\right) ,$ $k=1,2,...$ Then the following statements are equivalent to each
other:

\begin{enumerate}
\item $\underset{k\rightarrow \infty }{\lim }\left\Vert u_{k}-u\right\Vert
_{L^{p\left( x\right) }\left( \Omega \right) }=0$;

\item $\underset{k\rightarrow \infty }{\lim }\sigma _{p}\left(
u_{k}-u\right) =0$;

\item $u_{k}$ converges to $u$ in $\Omega $ in measure and $\underset{%
k\rightarrow \infty }{\lim }\sigma _{p}\left( u_{k}\right) =\sigma
_{p}\left( u\right) .$
\end{enumerate}
\end{lemma}

Let $\Omega \subset \mathbb{R}^{n}$ be a bounded domain and $p\in L^{\infty
}\left( \Omega \right) $ then variable exponent Sobolev space is defined as, 
\begin{equation*}
W^{1,\text{ }p\left( x\right) }\left( \Omega \right) :=\left\{ u\in
L^{p\left( x\right) }\left( \Omega \right) :\text{ }\left\vert \nabla
u\right\vert \in L^{p\left( x\right) }\left( \Omega \right) \right\}
\end{equation*}%
and this space is a separable Banach space with the norm 
\begin{equation*}
\left\Vert u\right\Vert _{W^{1,\text{ }p\left( x\right) }\left( \Omega
\right) }\equiv \left\Vert u\right\Vert _{L^{p\left( x\right) }\left( \Omega
\right) }+\left\Vert \nabla u\right\Vert _{L^{p\left( x\right) }\left(
\Omega \right)}.
\end{equation*}

\medskip

In the following discussion, we give the definition of generalized nonlinear
spaces (functional pn-spaces with variable exponent) and features of them
that indicate their relation with known spaces. These classes are nonlinear
spaces which are generalization of nonlinear spaces with constant exponent
studied in \cite{S4} (see also references therein). We also specify that
some of the results and its proofs can be found in \cite{S-S-1,S-S-2}.

\begin{definition}
\label{uz} Let $\Omega \subset \mathbb{R}^{n}\left( n\geq 2\right) $ be a
bounded domain with Lipschitz boundary and $\gamma ,$ $\beta $ $\in
M_{0}\left( \Omega \right) .$ We introduce $S_{1,\gamma \left( x\right)
,\beta \left( x\right) }\left( \Omega \right) ,$ the class of functions $%
u:\Omega \rightarrow \mathbb{R}$ and the functional $[.]_{S_{\gamma ,\beta
}}:S_{1,\gamma \left( x\right) ,\beta \left( x\right) }\left( \Omega \right)
\longrightarrow \mathbb{R}_{+}$ as follows:%
\begin{equation*}
S_{1,\gamma \left( x\right) ,\beta \left( x\right) }\left( \Omega \right)
:=\left\{ u\in L^{1}\left( \Omega \right) :\int\limits_{\Omega }\left\vert
u\right\vert ^{\gamma \left( x\right) +\beta \left( x\right)
}dx+\sum_{i=1}^{n}\int\limits_{\Omega }\left\vert u\right\vert ^{\gamma
\left( x\right) }\left\vert D_{i}u\right\vert ^{\beta \left( x\right)
}dx<\infty \right\} ,
\end{equation*}%
\begin{equation*}
\lbrack u]_{S_{\gamma ,\beta }}:=\inf \left\{ \lambda >0:\int\limits_{\Omega
}\left\vert \frac{u}{\lambda }\right\vert ^{\gamma \left( x\right) +\beta
\left( x\right) }dx+\sum_{i=1}^{n}\left( \int\limits_{\Omega }\left\vert 
\frac{\left\vert u\right\vert ^{\frac{\gamma \left( x\right) }{\beta \left(
x\right) }}D_{i}u}{\lambda ^{\frac{\gamma \left( x\right) }{\beta \left(
x\right) }+1}}\right\vert ^{\beta \left( x\right) }\right) dx\leq 1\right\} .
\end{equation*}
\end{definition}

$[.]_{S_{\gamma ,\beta }}$ defines a pseudo-norm on $S_{1,\gamma \left(
x\right) ,\beta \left( x\right) }\left( \Omega \right) ,$ actually it can be
readily verified that $[.]_{S_{\gamma ,\beta }}$ fulfills all axioms of
pseudo-norm (pn) see \cite{S-A}, \cite{S-Sp} i.e. $[u]_{S_{\gamma ,\beta
}}\geq 0,$ $u=0\Rightarrow \lbrack u]_{S_{\gamma ,\beta }}=0,$ $%
[u]_{S_{\gamma ,\beta }}\neq \lbrack v]_{S_{\gamma ,\beta }}\Rightarrow
u\neq v$ and $[u]_{S_{\gamma ,\beta }}=0\Rightarrow u=0.$

\medskip

Let $S_{1,\gamma \left( x\right) ,\beta \left( x\right) }\left( \Omega
\right) $ be the space given in the Definition \ref{uz} and $\theta \left(
x\right)\in M_{0}\left( \Omega \right)$, we denote $S_{1,\gamma \left(
x\right) ,\beta \left( x\right) ,\theta \left( x\right) }\left( \Omega
\right) ,$ the class of functions $u:\Omega \rightarrow \mathbb{R} $ by the
following intersection: 
\begin{equation*}  \label{arakesit}
S_{1,\gamma \left( x\right) ,\beta \left( x\right) ,\theta \left( x\right)
}\left( \Omega \right) :=S_{1,\gamma \left( x\right) ,\beta \left( x\right)
}\left( \Omega \right) \cap L^{\theta \left( x\right) }\left( \Omega \right)
,  \tag{2.6}
\end{equation*}%
with the pseudo-norm 
\begin{equation*}
\lbrack u]_{S_{\gamma ,\beta ,\theta} }:=[u]_{S_{\gamma ,\beta }}+\left\Vert
u\right\Vert _{L^{\theta \left( x\right) }\left( \Omega \right) },\text{ \ }%
\forall u\in S_{1,\gamma \left( x\right) ,\beta \left( x\right) ,\theta
\left( x\right) }\left( \Omega \right) .
\end{equation*}

\begin{proposition}
\label{oner} If $\gamma $, $\beta $, $\theta $ $\in M_{0}\left( \Omega
\right) $ and $\theta \left( x\right) \geq \gamma \left( x\right) +\beta
\left( x\right) +\varepsilon _{0}$ a.e. $x\in \Omega $ for some $\varepsilon
_{0}>0,$ then we have the following equivalence; 
\begin{equation*}
S_{1,\gamma \left( x\right) ,\beta \left( x\right) ,\theta \left( x\right)
}\left( \Omega \right) \equiv \left\{ u\in L^{1}\left( \Omega \right) :\Re
^{\gamma ,\beta ,\theta }\left( u\right)<\infty \right\}
\end{equation*}
where $\Re ^{\gamma ,\beta ,\theta }\left( u\right) :=\int\limits_{\Omega
}\left\vert u\right\vert ^{\theta \left( x\right)
}dx+\sum_{i=1}^{n}\int\limits_{\Omega }\left\vert u\right\vert ^{\gamma
\left( x\right) }\left\vert D_{i}u\right\vert ^{\beta \left( x\right)}dx,$ 
\newline
and the pseudo-norm on this space is 
\begin{equation*}
\lbrack u]_{S_{\gamma ,\beta ,\theta }}\equiv \inf \left\{ \lambda
>0:\int\limits_{\Omega }\left\vert \frac{u}{\lambda }\right\vert ^{\theta
\left( x\right) }dx+\sum_{i=1}^{n}\left( \int\limits_{\Omega }\left\vert 
\frac{\left\vert u\right\vert ^{\frac{\gamma \left( x\right) }{\beta \left(
x\right) }}D_{i}u}{\lambda ^{\frac{\gamma \left( x\right) }{\beta \left(
x\right) }+1}}\right\vert ^{\beta \left( x\right) }\right) dx\leq 1\right\} .
\end{equation*}
\end{proposition}

\begin{lemma}
\label{norm} Assume that conditions of Proposition \ref{oner} are fulfilled.
Let $u\in $ $S_{1,\gamma \left( x\right) ,\beta \left( x\right) ,\theta
\left( x\right) }\left( \Omega \right) $ and $\lambda _{u}:=[u]_{S_{\gamma
,\beta ,\theta }},$ then the following inequality%
\begin{equation*}
\max \{\lambda _{u}^{\gamma ^{-}+\beta ^{-}},\lambda _{u}^{\theta
^{+}}\}\geq \Re ^{\gamma ,\beta ,\theta }\left( u\right) \geq \min \{\lambda
_{u}^{\gamma ^{-}+\beta ^{-}},\lambda _{u}^{\theta ^{+}}\}
\end{equation*}%
holds.
\end{lemma}

\begin{theorem}
\label{gom} Suppose that conditions of Proposition \ref{oner} are satisfied
and let $p \in M_{0}\left( \Omega \right),$ $p\left( x\right) \geq \theta
\left( x\right) $ a.e. $x\in \Omega .$ Then, the embedding 
\begin{equation}  \label{gomtag}
W^{1,\text{ }p\left( x\right) }\left( \Omega \right) \subset S_{1,\gamma
\left( x\right) ,\beta \left( x\right) ,\theta \left( x\right) }\left(
\Omega \right)  \tag{2.7}
\end{equation}
holds.
\end{theorem}

\begin{definition}
\label{notban} Let $\eta $ $\in M_{0}\left( \Omega \right),$ we introduce $%
L^{1,\text{ }\eta \left( x\right) }\left( \Omega \right)$ the class of
functions $u:\Omega \rightarrow \mathbb{R} $ 
\begin{equation*}
L^{1,\text{ }\eta \left( x\right) }\left( \Omega \right) \equiv \left\{ u\in
L^{1}\left( \Omega \right) |\text{ }D_{i}u\in L^{\eta \left( x\right)
}\left( \Omega \right) ,\text{ }i=\overline{1,n}\right\}.\footnote{%
This space is not Banach differently from the space $W^{1,\text{ }\eta
\left( x\right) }\left( \Omega \right)$ \cite{DIE2} }
\end{equation*}
\end{definition}

\begin{theorem}
\label{bijec} Let $\gamma ,$ $\beta \in M_{0}\left( \Omega \right)\cap
C^{1}\left( \bar{\Omega}\right) $ and $L^{1,\text{ }\beta \left( x\right)
}\left( \Omega \right) $ be the space given in Definition \ref{notban}. Then
the function $\varphi :\Omega \times \mathbb{R} \longrightarrow \mathbb{R},$ 
$\varphi \left( x,t\right) :=\left\vert t\right\vert ^{\frac{\gamma \left(
x\right) }{\beta \left( x\right) }}t$ is a bijective mapping between $%
S_{1,\gamma \left( x\right) ,\beta \left( x\right) ,\theta \left( x\right)
}\left( \Omega \right) $ and $L^{1,\text{ }\beta \left( x\right) }\left(
\Omega \right) \cap L^{\psi \left( x\right) }\left( \Omega \right) $ where $%
\psi \left( x\right) :=\frac{\theta \left( x\right) \beta \left( x\right) }{%
\gamma \left( x\right) +\beta \left( x\right) }.$
\end{theorem}

\begin{theorem}
\label{kompakt} Suppose that conditions of Theorem \ref{bijec} are
satisfied. Let $p\in M_{0}\left( \Omega \right) ,$ additionally $1\leq \beta
^{-}\leq \beta \left( x\right) <n,$ $x\in \Omega $ holds and for $%
\varepsilon >0,$ the inequality 
\begin{equation*}
p\left( x\right) +\varepsilon <\tfrac{n\left( \gamma \left( x\right) +\beta
\left( x\right) \right) }{n-\beta \left( x\right) },\text{ }x\in \Omega
\end{equation*}%
is satisfied. Then the following compact embedding 
\begin{equation*}
S_{1,\gamma \left( x\right) ,\beta \left( x\right) ,\theta \left( x\right)
}\left( \Omega \right) \hookrightarrow L^{p\left( x\right) }\left( \Omega
\right)
\end{equation*}
exists.
\end{theorem}

\section{Some relations between constant exponent pn-spaces and Sobolev
spaces}

In this section, we give some embedding results for constant exponent $pn$%
-spaces with proofs.

\begin{theorem}
\label{sbtgom1} Let $\alpha \geq 0,$ $\beta \geq 1.$ Then for all $p$
satisfying the followings conditions

\begin{itemize}
\item[(i)] If $\beta =n,$ $p>\beta .$

\item[(ii)] If $\beta >n,$ $p\geq \beta .$

\item[(iii)] If $\beta <n,$ $p\geq \frac{n(\alpha +\beta )}{\alpha +n}, $
\end{itemize}

the embedding 
\begin{equation}  \label{3.1}
W_{0}^{1,p}(\Omega )\subset \mathring{S}_{1,\alpha ,\beta }(\Omega ). 
\tag{3.1}
\end{equation}%
holds.
\end{theorem}

\begin{proof}
The cases (i) and (ii) are evident as by virtue of the Sobolev imbedding
theorems occurs the inclusion 
\begin{equation*}
W_{0}^{1,p}(\Omega )\subset C(\bar{\Omega}).
\end{equation*}

For the last case (iii), if $\beta <n$ and $p>n$ then the proof is same with
the proofs of the cases (i) and (ii).

On the other side let $\beta <n$ and $p\in \left[ \frac{n(\alpha +\beta )}{%
\alpha +n},n\right) ,$ by Sobolev imbedding theorems we have, 
\begin{equation*}  \label{3.2}
W_{0}^{1,p}(\Omega )\subset L^{\tilde{q}}(\Omega )  \tag{3.2}
\end{equation*}%
for all $\tilde{q}\in \left[ 1,\frac{np}{n-p}\right] .$ Hence for $u\in
W_{0}^{1,p}(\Omega )$ we have the following estimate by Young's inequality 
\begin{equation}  \label{3.3}
\int\limits_{\Omega }\left\vert u\right\vert ^{\alpha }\left\vert
D_{i}u\right\vert ^{\beta }dx\leq \left( \frac{p-\beta }{p}\right)
\int\limits_{\Omega }\left\vert u\right\vert ^{\frac{\alpha p}{p-\beta }%
}dx+\left( \frac{p}{\beta }\right) \int\limits_{\Omega }\left\vert
D_{i}u\right\vert ^{p}dx.  \tag{3.3}
\end{equation}%
We deduce from the equation $\frac{\alpha p}{p-\beta }-\frac{np}{n-p}=\frac{p%
\left[ n(\alpha +\beta )-p(\alpha +n)\right] }{(p-\beta )(n-p)}$ and $p\in %
\left[ \frac{n(\alpha +\beta )}{\alpha +n},n\right) $ that 
\begin{equation*}
\frac{\alpha p}{p-\beta }\leq \frac{np}{n-p}
\end{equation*}%
Thus by \eqref{3.2} and \eqref{3.3} we arrive at 
\begin{equation*}
\left[ u\right] _{\mathring{S}_{1,\alpha ,\beta }}^{\alpha +\beta
}=\int\limits_{\Omega }\left\vert u\right\vert ^{\alpha }\left\vert
D_{i}u\right\vert ^{\beta }dx\leq \tilde{C}\left\Vert u\right\Vert
_{W_{0}^{1,p}(\Omega )}^{\frac{\alpha p}{p-\beta }}+\tilde{C}_{1}\left\Vert
u\right\Vert _{W_{0}^{1,p}(\Omega )}^{p}
\end{equation*}%
which implies 
\begin{equation*}
\left[ u\right] _{\mathring{S}_{1,\alpha ,\beta }}^{\alpha +\beta }\leq 
\tilde{C}_{2}\left\Vert u\right\Vert _{W_{0}^{1,p}(\Omega )}^{p}+C_{3}.
\end{equation*}%
To complete the proof if $p=n>\beta ,$ by employing the embedding $%
W_{0}^{1,p}(\Omega )\subset L^{r}(\Omega ),$ $1\leq r<\infty $ one can
obtain the desired result by the help of above approach.
\end{proof}

\begin{remark}
Under the conditions of Theorem \ref{sbtgom1}, if $p\geq \alpha +\beta $ is
satisfied then we have the imbedding \eqref{3.1} independently from
dimension of $\Omega $.
\end{remark}

Actually for $u\in W_{0}^{1,p}(\Omega )$, we deduce from Lemma \ref{sbt2}
that 
\begin{equation*}
\int\limits_{\Omega }\left\vert u\right\vert ^{\alpha }\left\vert
D_{i}u\right\vert ^{\beta }dx\leq C\int\limits_{\Omega }\left\vert
D_{i}u\right\vert ^{p}dx+C_{1},
\end{equation*}
which yields 
\begin{equation*}
\left[ u\right] _{\mathring{S}_{1,\alpha ,\beta }}^{\alpha +\beta }\leq
C\left\Vert u\right\Vert _{W_{0}^{1,p}(\Omega )}^{p}+C_{1}.
\end{equation*}

\begin{theorem}
Suppose that $\beta>\alpha\geq 0,$ $\beta \geq 2.$ Then for all $p$
satisfying the followings

\begin{itemize}
\item[(i)] If $\alpha +\beta=n $ then $1\leq p<2\beta $

\item[(ii)] If $\alpha +\beta>n $ then $1\leq p\leq 2\beta $

\item[(iii)] If $\alpha +\beta<n $ then $1\leq p\leq \frac{2n\beta (\alpha
+\beta )}{2n\beta -(\alpha +\beta )(\beta -\alpha )}$
\end{itemize}

the embedding 
\begin{equation}
\mathring{S}_{2,\alpha ,\beta }\left( \Omega \right) \subset
W_{0}^{1,p}(\Omega )  \tag{3.4}
\end{equation}%
holds.
\end{theorem}

\begin{proof}
Considering these conditions, by Lemma \ref{sbt3} when $1\leq p\leq \alpha
+\beta $ following inequality holds independently from the dimension $n$ 
\begin{equation*}
\int\limits_{\Omega }\left\vert D_{i}u\right\vert ^{p}dx\leq
C\int\limits_{\Omega }\left\vert u\right\vert ^{\alpha }\left\vert
D_{i}^{2}u\right\vert ^{\beta }dx+C_{1}  \tag{3.5}
\end{equation*}%
that yields the imbedding (3.4). So if $1\leq p\leq 2$ then $1\leq p\leq
\alpha +\beta $ which concludes the proof.

First we prove (3.4) in line with conditions of (i). Let $\alpha +\beta =n$
and $p>2$ (from now on we assume $p>2$)

For $u\in \mathring{S}_{2,\alpha ,\beta }\left( \Omega \right) ,$ by Lemma %
\ref{sbt3} we have the following estimate 
\begin{equation}
\int\limits_{\Omega }\left\vert D_{i}u\right\vert ^{\alpha +\beta }dx\leq
C\int\limits_{\Omega }\left\vert u\right\vert ^{\alpha }\left\vert
D_{i}^{2}u\right\vert ^{\beta }dx.  \tag{3.6}
\end{equation}%
On the other hand from Sobolev imbedding theorems, 
\begin{equation}
W_{0}^{1,\alpha +\beta }\left( \Omega \right) \subset L^{q}\left( \Omega
\right) \text{ }\forall q,\text{ }q\in \left[ 1,\infty \right)  \tag{3.7}
\end{equation}
Hence from (3.6) and (3.7) for all $q$ satisfying $1\leq q<\infty $ we get 
\begin{align*}
\left\Vert u\right\Vert _{q}& \leq \tilde{C}\left( \sum_{i=1}^{n}\left\Vert
D_{i}u\right\Vert _{\alpha +\beta }^{\alpha +\beta }\right) ^{\frac{1}{%
\alpha +\beta }} \\
& \leq \tilde{C}_{0}\left( \sum_{i=1}^{n}\left[ \int\limits_{\Omega
}\left\vert u\right\vert ^{\alpha }\left\vert D_{i}^{2}u\right\vert ^{\beta
}dx\right] \right) ^{\frac{1}{\alpha +\beta }} \\
& =\tilde{C}_{0}\left[ u\right] _{\mathring{S}_{2,\alpha ,\beta }}. 
\tag{3.8}
\end{align*}%
Therefore for all $u\in \mathring{S}_{2,\alpha ,\beta }\left( \Omega \right) 
$ and $i=1..n,$ 
\begin{align*}
\int\limits_{\Omega }\left\vert D_{i}u\right\vert ^{p}dx&
=\int\limits_{\Omega }\left( D_{i}u\left\vert D_{i}u\right\vert
^{p-2}\right) D_{i}udx \\
& =\left( p-1\right) \int\limits_{\Omega }uD_{i}^{2}u\left\vert
D_{i}u\right\vert ^{p-2}dx \\
& \leq \left( p-1\right) \int\limits_{\Omega }\left\vert u\right\vert ^{%
\frac{\beta -\alpha }{\beta }}\left\vert u\right\vert ^{\frac{\alpha }{\beta 
}}\left\vert D_{i}^{2}u\right\vert \left\vert D_{i}u\right\vert ^{p-2}dx. 
\tag{3.9}
\end{align*}%
Employing H\"{o}lder's inequality in (3.9) with exponents $\left( \frac{%
p\beta }{2\beta -p},\beta ,\frac{p}{p-2}\right) $ we obtain 
\begin{align*}
\int\limits_{\Omega }\left\vert D_{i}u\right\vert ^{p}dx& \leq C\left(
\int\limits_{\Omega }\left\vert u\right\vert ^{\frac{p(\beta -\alpha )}{%
2\beta -p}}dx\right) ^{\frac{2\beta -p}{p\beta }}\left( \int\limits_{\Omega
}\left\vert u\right\vert ^{\alpha }\left\vert D_{i}^{2}u\right\vert ^{\beta
}dx\right) ^{\frac{1}{\beta }}\left( \int\limits_{\Omega }\left\vert
D_{i}u\right\vert ^{p}dx\right) ^{\frac{p-2}{p}} \\
& =C\left\Vert u\right\Vert _{\frac{p(\beta -\alpha )}{2\beta -p}}^{\frac{%
\beta -\alpha }{\beta }}\left[ u\right] _{\mathring{S}_{2,\alpha ,\beta }}^{%
\frac{\alpha +\beta }{\beta }}\left\Vert D_{i}u\right\Vert _{p}^{p-2} 
\tag{3.10}
\end{align*}%
Estimating (3.10) by using (3.8) we get, 
\begin{align*}
\int\limits_{\Omega }\left\vert D_{i}u\right\vert ^{p}dx& \leq \tilde{C}%
\left[ u\right] _{\mathring{S}_{2,\alpha ,\beta }}^{\frac{\beta -\alpha }{%
\beta }}\left[ u\right] _{\mathring{S}_{2,\alpha ,\beta }}^{\frac{\alpha
+\beta }{\beta }}\left\Vert D_{i}u\right\Vert _{p}^{p-2} \\
& =\tilde{C}\left[ u\right] _{\mathring{S}_{2,\alpha ,\beta }}^{2}\left\Vert
D_{i}u\right\Vert _{p}^{p-2}.  \tag{3.11}
\end{align*}%
By using Young's inequality in (3.11), we arrive at 
\begin{equation*}
\left\Vert D_{i}u\right\Vert _{p}^{p}\leq \tilde{C}\left( \varepsilon
\right) \left[ u\right] _{\mathring{S}_{2,\alpha ,\beta }}^{p}+\tilde{C}%
\varepsilon \left\Vert D_{i}u\right\Vert _{p}^{p},
\end{equation*}%
choosing $\varepsilon $ such that $\tilde{C}\varepsilon <1$ then we acquire 
\begin{equation*}
\left\Vert D_{i}u\right\Vert _{p}\leq \tilde{C}\left[ u\right] _{\mathring{S}%
_{2,\alpha ,\beta }}<\infty
\end{equation*}%
which completes the proof for the case (i).

Assume that (ii) holds i.e. $\alpha +\beta>n $ and $2<p\leq 2\beta.$ Then 
\begin{equation*}
W^{1,\alpha +\beta }\left( \Omega \right) \subset C\left( \bar{\Omega}\right)
\tag{3.12}
\end{equation*}%
By (3.6) and (3.8), we achieve 
\begin{equation}
\left\Vert u\right\Vert _{C\left( \bar{\Omega}\right) }\leq \tilde{C}\left[ u%
\right] _{\mathring{S}_{2,\alpha ,\beta }}.  \tag{3.13}
\end{equation}
For all $u\in\mathring{S}_{2,\alpha ,\beta }\left( \Omega \right)$ from
(3.9) one concludes, 
\begin{align*}
\left\Vert D_{i}u\right\Vert _{p}^{p}&\leq \left( p-1\right)
\int\limits_{\Omega }\left\vert u\right\vert ^{\frac{\beta -\alpha }{\beta }%
}\left\vert u\right\vert ^{\frac{\alpha }{\beta }}\left\vert
D_{i}^{2}u\right\vert \left\vert D_{i}u\right\vert ^{p-2}dx \\
&\leq (p-1)C(\varepsilon )\int\limits_{\Omega }\left\vert u\right\vert
^{\beta -\alpha }\left\vert u\right\vert ^{\alpha }\left\vert
D_{i}^{2}u\right\vert ^{\beta }dx+(p-1)\varepsilon \int\limits_{\Omega
}\left\vert D_{i}u\right\vert ^{\frac{\beta (p-2)}{\beta -1}}dx \\
&\leq (p-1)C\left( \varepsilon \right) \left\Vert u\right\Vert _{C\left( 
\bar{\Omega}\right) }^{\beta -\alpha }\int\limits_{\Omega }\left\vert
u\right\vert ^{\alpha }\left\vert D_{i}^{2}u\right\vert ^{\beta
}dx+(p-1)\varepsilon \left\Vert D_{i}u\right\Vert _{\frac{\beta (p-2)}{\beta
-1}}^{\frac{\beta (p-2)}{\beta-1}}
\end{align*}
By using (3.13) and $\frac{\beta (p-2)}{\beta -1}-p=\frac{p-2\beta }{\beta -1%
}$ with $p\leq 2\beta $ to estimate $\left\Vert u\right\Vert _{C\left( \bar{%
\Omega}\right) }^{\beta -\alpha }$ and $\left\Vert D_{i}u\right\Vert _{\frac{%
\beta (p-2)}{\beta -1}}^{\frac{\beta (p-2)}{\beta -1}}$ respectively, we
arrive at 
\begin{align*}
\left\Vert D_{i}u\right\Vert _{p}^{p}&\leq C\left( \varepsilon \right) (p-1) 
\left[ u\right] _{\mathring{S}_{2,\alpha ,\beta }}^{\beta -\alpha }\left[ u%
\right] _{\mathring{S}_{2,\alpha ,\beta }}^{\alpha +\beta }+(p-1)\varepsilon 
\tilde{C}C\left\Vert D_{i}u\right\Vert _{p}^{p}+(p-1)\varepsilon C_{1} \\
&=C\left( \varepsilon \right) \left[ u\right] _{\mathring{S}_{2,\alpha
,\beta }}^{2\beta }+\varepsilon \tilde{C}C\left\Vert D_{i}u\right\Vert
_{p}^{p}+\varepsilon C_{1}
\end{align*}
which implies 
\begin{equation*}
\left\Vert D_{i}u\right\Vert _{p}^{p}\leq \tilde{C}\left[ u\right] _{%
\mathring{S}_{2,\alpha ,\beta }}^{2\beta }+C_{1}
\end{equation*}
that ends the proof.

For the last case (iii), let\ $\alpha +\beta <n$ and $1\leq p\leq \frac{
2n\beta (\alpha +\beta )}{2n\beta -(\alpha +\beta )(\beta -\alpha )}.$ From
Sobolev imbedding theorems 
\begin{equation*}
W^{1,\alpha +\beta }\left( \Omega \right) \subset L^{\tilde{q}}\left( \Omega
\right) \text{ }\forall \tilde{q},\text{ }\tilde{q}\in \left[ 1,\frac{%
n\left( \alpha +\beta \right) }{n-\left( \alpha +\beta \right) }\right] 
\tag{3.14}
\end{equation*}%
By (3.6) and (3.14), we attain 
\begin{equation}
\left\Vert u\right\Vert _{\tilde{q}}\leq C\left[ u\right] _{\mathring{S}%
_{2,\alpha ,\beta }}  \tag{3.15}
\end{equation}%
For all $u\in \mathring{S}_{2,\alpha ,\beta }\left( \Omega \right),$ we
deduce from the inequality $p\leq \frac{2n\beta (\alpha +\beta )}{2n\beta
-(\alpha +\beta)(\beta -\alpha )}<2\beta $ that 
\begin{equation}
\left\Vert D_{i}u\right\Vert _{p}^{p}\leq C\left\Vert u\right\Vert _{\frac{%
p(\beta -\alpha )}{2\beta -p}}^{\frac{\beta -\alpha }{\beta }}\left[ u\right]
_{\mathring{S}_{2,\alpha ,\beta }}^{\frac{\alpha +\beta }{\beta }}\left\Vert
D_{i}u\right\Vert _{p}^{p-2}.  \tag{3.16}
\end{equation}%
If we take the inequality $\frac{p(\beta-\alpha )}{2\beta-p}\leq \frac{%
n\left( \alpha+\beta \right) }{n-\left( \alpha+\beta \right) }$ into account
and estimate $\left\Vert u\right\Vert _{\frac{p(\beta-\alpha )}{2\beta-p}}$
in (3.16) by (3.15) we obtain, 
\begin{align*}  \label{3.17}
\left\Vert D_{i}u\right\Vert _{p}^{p}& \leq \tilde{C}\left[ u\right] _{%
\mathring{S}_{2,\alpha ,\beta }}^{\frac{\beta -\alpha }{\beta }}\left[ u%
\right] _{\mathring{S}_{2,\alpha ,\beta }}^{\frac{\alpha +\beta }{\beta }%
}\left\Vert D_{i}u\right\Vert _{p}^{p-2} \\
& =\tilde{C}\left[ u\right] _{\mathring{S}_{2,\alpha ,\beta }}^{2}\left\Vert
D_{i}u\right\Vert _{p}^{p-2}  \tag{3.17}
\end{align*}%
Applying Young's inequality in \eqref{3.17} we attain, 
\begin{equation*}
\left\Vert D_{i}u\right\Vert _{p}^{p}\leq \tilde{C}\left( \varepsilon
\right) \left[ u\right] _{\mathring{S}_{2,\alpha ,\beta }}^{p}+\tilde{C}%
\varepsilon \left\Vert D_{i}u\right\Vert_{p}^{p}
\end{equation*}%
that yields 
\begin{equation*}
\left\Vert D_{i}u\right\Vert _{p}\leq \tilde{C}\left[ u\right] _{\mathring{S}%
_{2,\alpha ,\beta }}
\end{equation*}
so the proof is complete.
\end{proof}

\medskip We now turn our attention to some examples and results for one
dimensional case:

\begin{definition}
\label{birtan} \label{birtan}Let $\alpha >\beta -1\geq 0$ we define the
following function space: 
\begin{equation*}
\tilde{S}_{2,\alpha ,\beta }(a,b):=\{u\in \text{ }L^{1}(a,b)\mid \left[ u%
\right] _{\tilde{S}_{1,\alpha ,\beta }(a,b)}^{\alpha +\beta
}=\int\limits_{a}^{b}\left\vert u\right\vert ^{\alpha +\beta
}dx+\int\limits_{a}^{b}\left\vert u\right\vert ^{\alpha -\beta }\left\vert
Du\right\vert ^{2\beta }dx
\end{equation*}%
\begin{equation*}
+\int\limits_{a}^{b}\left\vert u\right\vert ^{\alpha }\left\vert
D^{2}u\right\vert ^{\beta }dx<\infty \}.
\end{equation*}
\end{definition}

The proofs of the following lemmas can be attained readily thus we skip the
proofs for the sake of brevity.

\begin{lemma}
\label{birgom} Let $\tilde{S}_{2,\alpha ,\beta }(a,b)$ be the space given in
Definition \ref{birtan}, then the imbedding 
\begin{equation*}
\tilde{S}_{2,\alpha ,\beta }(a,b)\subset S_{1,\alpha ,\beta }(a,b).
\end{equation*}
holds.
\end{lemma}

\begin{lemma}
\label{birhom}Let $\alpha >\beta -1>0$ and $g(t)\equiv \left\vert
t\right\vert ^{\frac{\alpha }{\beta }}t$ for any $t\in R^{1}$. Then
following assertions are true

\begin{itemize}
\item[1)] If $u\in $ $\tilde{S}_{2,\alpha ,\beta }(a,b)$ then $g\left(
u\right) \in W^{2,\beta }(a,b)$;

\item[2)] If a function $u\in L^{1}(a,b)$ such, that $g\left( u\right)
\equiv v\in W^{2,\beta }(a,b)$ then $u\in $ $\tilde{S}_{2,\alpha ,\beta
}(a,b)$.
\end{itemize}
\end{lemma}

Consequently, we can define the space $\tilde{S}_{2,\alpha ,\beta }(a,b)$ in
the following way by virtue of the general definition of the nonlinear spaces

\begin{definition}
\label{birgentan}Let $g:\mathbb{R}\rightarrow \mathbb{R},$ $g(t)=\left\vert
t\right\vert ^{\frac{\alpha }{\beta }}t$ and $\alpha >\beta -1>0$ then $%
\tilde{S}_{2,\alpha ,\beta }(a,b)$ has the following representation 
\begin{equation*}
\tilde{S}_{2,\alpha ,\beta }(a,b)=\left\{ u\in L^{1}(a,b)\mid \left[ u\right]
_{S_{gW^{2,\beta }}}^{\alpha +\beta }\equiv \sum\limits_{0\leq s\leq
2}\left\Vert D^{s}g(u)\right\Vert _{\beta }^{\beta }<\infty \right\} \equiv
S_{gW^{2,\beta }}(a,b).
\end{equation*}
\end{definition}

\begin{remark}
\label{birremark}The following equivalences are true 
\begin{equation*}
\tilde{S}_{2,\alpha ,\beta }(a,b)\cap \left\{ u\left\vert \ u\left\vert \
_{\partial \Omega }=0\right. \right. \right\} \equiv \overset{0}{S}%
_{2,\alpha ,\beta }(a,b)
\end{equation*}%
and 
\begin{equation*}
\sum\limits_{0\leq s\leq k}\left\Vert D^{s}g(u)\right\Vert _{\beta }^{\beta
}\equiv \sum\limits_{0\leq s\leq k}\left\Vert g^{-1}\left( D^{s}g(u)\right)
\right\Vert _{\alpha +\beta }^{\alpha +\beta }
\end{equation*}%
for $k=0,1$, but for $k=2$ 
\begin{equation*}
\left\Vert g^{\prime }(u)D^{2}u\right\Vert _{\beta }^{\beta }\equiv
\left\Vert g^{-1}\left( g^{\prime }(u)D^{2}u)\right) \right\Vert _{\alpha
+\beta }^{\alpha +\beta }\quad \&
\end{equation*}%
\begin{equation*}
\left\Vert g^{\prime \prime }(u)\left( Du\right) ^{2}\right\Vert _{\beta
}^{\beta }\equiv \left\Vert g^{-1}\left( g^{\prime \prime }(u)\left(
Du\right) ^{2})\right) \right\Vert _{\alpha +\beta }^{\alpha +\beta }.
\end{equation*}
\end{remark}

The following example shows the nonlinear structure of the pn-spaces.

\begin{example}
\label{sbtornek2} Let $\beta >1.$ Then $S_{1,1,\beta }(0,1)$ is a nonlinear
space.
\end{example}

Let $\tau \in \left( \frac{\beta -1}{\beta +1},\frac{\beta -1}{\beta }\right]
$ and define the functions 
\begin{equation*}
u_{0}\left( x\right):= x^{\tau }\text{ and }u_{1}\left( x\right) :=\theta 
\text{, }x\in \left( 0,1\right) \text{, }\left( \theta \in \mathbb{R}^{+}%
\text{is a constant.}\right)
\end{equation*}
It is easy to show that $u_{0},$ $u_{1}\in S_{1,1,\beta }(0,1)$ by the
definition of $S_{1,1,\beta }(0,1).$ Besides $u\left( x\right):=u_{0}\left(
x\right)+u_{1}\left( x\right)=x^{\tau}+\theta\not\in S_{1,1,\beta }(0,1).$ 
\begin{align*}
\left[ u\right]_{S_{1,1,\beta
}(0,1)}^{\beta+1}&=\int\limits_{0}^{1}\left\vert u\right\vert ^{\beta
+1}dx+\int\limits_{0}^{1}\left\vert u\right\vert \left\vert Du\right\vert
^{\beta }dx \\
&=\int\limits_{0}^{1}\left( x^{\tau }+\theta \right) ^{\beta +1}dx+\tau
^{\beta }\int\limits_{0}^{1}\left( x^{\tau }+\theta \right) x^{\beta (\tau
-1)}dx \\
&=\int\limits_{0}^{1}\left( x^{\tau }+\theta \right) ^{\beta +1}dx+\tau
^{\beta }\int\limits_{0}^{1}\left( x^{\tau \left( \beta +1\right) -\beta
}+\theta x^{\beta (\tau -1)}\right)dx.
\end{align*}
Since $\beta(\tau-1)\leq-1$ so, the right and side of the above equation is
divergent which implies $u\not\in S_{1,1,\beta }(0,1).$

\section{Variable Exponent Nonlinear Spaces and Embedding Theorems}

In this section, we present certain new results with detailed proofs for
variable exponent pn-spaces mentioned in Section 2. First, we derive
integral inequalities (see, also \cite{S-S-2}) to understand the structure
of these spaces. Afterwards, we prove some lemmas and theorems on continuous
embeddings of these spaces and on topology of them. (Throughout this
section, we assume that $\Omega \subset \mathbb{R} ^{n}\left( n\geq 2\right) 
$ is a bounded domain with Lipschitz boundary.)

\begin{lemma}
\label{var1} Let $\alpha,$ $\beta$ $\in M_{0}(\Omega)$ and $\alpha \left(
x\right) \geq \beta \left( x\right) $ a.e. $x\in \Omega .$ Then the
inequality 
\begin{equation}
\int\limits_{\Omega }\left\vert u\right\vert ^{\beta \left( x\right) }dx\leq
\int\limits_{\Omega }\left\vert u\right\vert ^{\alpha \left( x\right)
}dx+\left\vert \Omega \right\vert ,\text{ \ \ }\forall u\in L^{\alpha \left(
x\right) }\left( \Omega \right)  \tag{4.1}
\end{equation}%
holds.
\end{lemma}

\begin{proof}
Let $\Omega _{1}:=\left\{ x\in \Omega : \alpha \left( x\right)=\beta \left(
x\right)\right\} $ and $\Omega _{2}:=\Omega \setminus \Omega _{1}.$ Hence%
\begin{equation*}
\int\limits_{\Omega }\left\vert u\right\vert ^{\beta \left( x\right)
}dx=\int\limits_{\Omega _{1}}\left\vert u\right\vert ^{\alpha \left(
x\right) }dx+\int\limits_{\Omega _{2}}\left\vert u\right\vert ^{\beta \left(
x\right) }dx.
\end{equation*}%
Estimating the second integral on the right member of the above equation by
utilizing Young inequality ($\alpha \left( x\right) >\beta \left( x\right) $
on $\Omega _{2},$), we achieve that 
\begin{equation*}
\int\limits_{\Omega }\left\vert u\right\vert ^{\beta \left( x\right) }dx\leq
\int\limits_{\Omega _{1}}\left\vert u\right\vert ^{\alpha \left( x\right)
}dx+\int\limits_{\Omega _{2}}\left( \frac{\beta \left( x\right) }{\alpha
\left( x\right) }\right) \left\vert u\right\vert ^{\alpha \left( x\right)
}dx+\int\limits_{\Omega _{2}}\left( \frac{\alpha \left( x\right) -\beta
\left( x\right) }{\alpha \left( x\right) }\right) dx,
\end{equation*}%
since $\frac{\beta \left( x\right) }{\alpha \left( x\right) }<1$ and $\frac{%
\alpha \left( x\right) -\beta \left( x\right) }{\alpha \left( x\right) }<1,$
for $x\in \Omega _{2}$ we deduce from the last inequality that 
\begin{equation*}
\int\limits_{\Omega }\left\vert u\right\vert ^{\beta \left( x\right) }dx\leq
\int\limits_{\Omega _{1}}\left\vert u\right\vert ^{\alpha \left( x\right)
}dx+\int\limits_{\Omega _{2}}\left\vert u\right\vert ^{\alpha \left(
x\right) }dx+\left\vert \Omega \right\vert
\end{equation*}%
\begin{equation*}
=\int\limits_{\Omega }\left\vert u\right\vert ^{\alpha \left( x\right)
}dx+\left\vert \Omega \right\vert.
\end{equation*}%
On the other side if $\alpha \left( x\right) =\beta \left( x\right) $ a.e. $%
x\in \Omega ,$ then (4.1) is clear.
\end{proof}

\begin{lemma}
\label{var2} Assume that $\zeta \in M_{0}(\Omega)$ and $\beta \geq 1,$ $%
\epsilon >0.$ Then for every $u\in L^{\zeta \left( x\right) +\epsilon
}\left( \Omega \right) $%
\begin{equation}
\int\limits_{\Omega }\left\vert u\right\vert ^{\zeta \left( x\right)
}\left\vert \ln \left\vert u\right\vert \right\vert ^{\beta }dx\leq
N_{1}\int\limits_{\Omega }\left\vert u\right\vert ^{\zeta \left( x\right)
+\epsilon }dx+N_{2}  \tag{4.2}
\end{equation}%
is satisfied. Here $N_{1}\equiv N_{1}\left( \epsilon ,\beta \right) >0$ and $%
N_{2}\equiv N_{2}\left( \epsilon ,\beta ,\left\vert \Omega \right\vert
\right) >0$ are constants.
\end{lemma}

\begin{proof}
Let us consider the function $f\left( t\right) =\left\vert t\right\vert
^{\epsilon }-\ln \left\vert t\right\vert $ for $t\in \mathbb{R}%
-\left\{0\right\} $. Since $f$ is an even function it is sufficient to
investigate only $f\left( t\right) =t^{\epsilon }-\ln t$, $t>0.$ It can be
readily shown that this function is decreasing on $\left( 0,\frac{1}{\sqrt[%
\epsilon ]{\epsilon }}\right] $ and increasing on the interval $\left[ \frac{%
1}{\sqrt[\epsilon ]{\epsilon }},\infty \right) .$ Also $f\nearrow \infty $
when $x\searrow 0$ and $x\nearrow \infty $ and $f\left( \frac{1}{\sqrt[%
\epsilon ]{\epsilon }}\right) =\frac{1}{\epsilon }\left( 1+\ln \epsilon
\right) $. Here we have two situations (i) if $\epsilon \in \left( \frac{1}{e%
},\infty \right) $ then $f\left( \frac{1}{\sqrt[\epsilon ]{\epsilon }}%
\right) >0$ (ii) if $\epsilon \in \left( 0,\frac{1}{e}\right] $ then $%
f\left( \frac{1}{\sqrt[\epsilon ]{\epsilon }}\right) \leq 0.$ For the first
case (i) $\forall t\in \left( 0,\infty \right) ,$ $f\left( t\right) >0$ or
equivalently $\ln t<t^{\epsilon }$. For the case (ii), the function $f$ has
two zeros say $m_{1}>0$ and $m_{2}>0$ and for $t\in \mathbb{R}^{+}-\left(
m_{1},m_{2}\right) $ it is obvious that $\ln t<t^{\epsilon }.$ For $t\in %
\left[ m_{1},m_{2}\right] ,$ $\exists N_{0}>1$ ($N_{0}\equiv N_{0}\left( 
\frac{1}{\sqrt[\epsilon ]{\epsilon }}\right) $) such that $\ln
t<N_{0}t^{\epsilon }.$ Hence the inequality $\ln t\leq N_{0}t^{\epsilon }$
will be satisfied on $\left( 0,\infty \right) .$ As a result from the cases
(i) and (ii) for arbitrary $\epsilon >0$ and $t\in \mathbb{R}-\left\{
0\right\} $, we have the inequality 
\begin{equation*}
\ln \left\vert t\right\vert \leq N_{0}\left( \epsilon \right) \left\vert
t\right\vert ^{\epsilon }
\end{equation*}%
that implies on the set $\left\{ x\in \Omega :\left\vert u\left( x\right)
\right\vert \geq 1\text{ }\right\} $ the inequality $\left\vert u\right\vert
^{\zeta \left( x\right) }\left\vert \ln \left\vert u\right\vert \right\vert
^{\beta }\leq \newline
\leq N_{0}\left( \epsilon ,\beta \right) \left\vert u\right\vert ^{\zeta
\left( x\right) +\epsilon }$ be fulfilled. Moreover, from $\underset{%
t\rightarrow 0^{+}}{\lim }t^{\epsilon }\left\vert \ln t\right\vert ^{\beta
}=0$ and for every fixed $x_{0}\in \Omega $, $\underset{t\rightarrow 0^{+}}{%
\lim }\frac{\left\vert t\right\vert ^{\zeta \left( x_{0}\right) }\left\vert
\ln \left\vert t\right\vert \right\vert ^{\beta }}{t^{\zeta \left(
x_{0}\right) +\epsilon }+1}=0$, we arrive at the inequality $\left\vert
u\right\vert ^{\zeta \left( x\right) -1}\left\vert u\right\vert \left\vert
\ln \left\vert u\right\vert \right\vert ^{\beta }$ $\leq \tilde{N}_{0}\left(
\left\vert u\right\vert ^{\zeta \left( x\right) +\epsilon }+1\right) $ on
the set $\left\{ x\in \Omega :\left\vert u\left( x\right) \right\vert <1%
\text{ }\right\} $ for some $\tilde{N}_{0}=\tilde{N}_{0}\left( \epsilon
,\beta \right) >0$. So the proof is complete by the combination of these
inequalities.
\end{proof}

\begin{lemma}
\label{var3} Let $\tilde{\varepsilon}>0$ and $\beta _{1}:\Omega \rightarrow
\lbrack \tilde{\varepsilon},$ $\infty )$ be a measurable function which
satisfy $\tilde{\varepsilon}\leq \beta _{1}^{-}\leq \beta _{1}\left(
x\right) \leq \beta _{1}^{+}<\infty $ and $\xi,$ $\beta $ $\in M_{0}(\Omega)$
then the inequality 
\begin{equation}  \label{4.3}
\int\limits_{\Omega }\left\vert u\right\vert ^{\xi \left( x\right)
}\left\vert \ln \left\vert u\right\vert \right\vert ^{\beta \left( x\right)
}dx\leq C_{1}\int\limits_{\Omega }\left\vert u\right\vert ^{\xi \left(
x\right) +\beta _{1}\left( x\right) }dx+C_{2},\ \forall u\in L^{\xi \left(
x\right) +\beta _{1}\left( x\right) }\left( \Omega \right)  \tag{4.3}
\end{equation}%
holds. Here $C_{1}\equiv C_{1}\left( \tilde{\varepsilon},\beta ^{+}\right)
>0 $ and $C_{2}\equiv C_{2}\left( \tilde{\varepsilon}, \beta ^{+},
\left\vert \Omega \right\vert \right) >0$ are constants.
\end{lemma}

\begin{proof}
For arbitrary $\gamma \in \left( 0,1\right) ,$ $\frac{\beta ^{+}+\gamma }{%
\beta \left( x\right) }>1$ by utilizing the Young's inequality with this
exponent to $\left\vert \ln \left\vert u\right\vert \right\vert ^{\beta
\left( x\right) }$ we achieve the following inequality,%
\begin{equation*}
\left\vert \ln \left\vert u\right\vert \right\vert ^{\beta \left( x\right)
}\leq \left\vert \ln \left\vert u\right\vert \right\vert ^{\beta ^{+}+\gamma
}+1,
\end{equation*}%
by multiplying each side of this inequality with $\left\vert u\right\vert
^{\xi \left( x\right) },$ we get 
\begin{equation*}
\left\vert u\right\vert ^{\xi \left( x\right) }\left\vert \ln \left\vert
u\right\vert \right\vert ^{\beta \left( x\right) }\leq \left\vert
u\right\vert ^{\xi \left( x\right) }\left\vert \ln \left\vert u\right\vert
\right\vert ^{\beta ^{+}+\gamma }+\left\vert u\right\vert ^{\xi \left(
x\right) },\text{ \ }x\in \Omega .
\end{equation*}%
Thus integrating both sides over $\Omega $,%
\begin{equation*}
\int\limits_{\Omega }\left\vert u\right\vert ^{\xi \left( x\right)
}\left\vert \ln \left\vert u\right\vert \right\vert ^{\beta \left( x\right)
}dx\leq \int\limits_{\Omega }\left\vert u\right\vert ^{\xi \left( x\right)
}\left\vert \ln \left\vert u\right\vert \right\vert ^{\beta ^{+}+\gamma
}dx+\int\limits_{\Omega }\left\vert u\right\vert ^{\xi \left( x\right) }dx
\end{equation*}%
is established. For $\epsilon <\tilde{\varepsilon},$ estimating the first
integral on the right side of the last inequality by Lemma \ref{var2}, we
acquire, 
\begin{equation*}
\int\limits_{\Omega }\left\vert u\right\vert ^{\xi \left( x\right)
}\left\vert \ln \left\vert u\right\vert \right\vert ^{\beta \left( x\right)
}dx\leq C_{3}\int\limits_{\Omega }\left\vert u\right\vert ^{\xi \left(
x\right) +\epsilon }dx+C_{4}+\int\limits_{\Omega }\left\vert u\right\vert
^{\xi \left( x\right) }dx.
\end{equation*}%
As $\frac{\xi \left( x\right) +\epsilon }{\xi \left( x\right) }>1,$ applying
Lemma \ref{var1} to estimate the second integral on the right member of the
last inequality, we gain 
\begin{equation*}
\int\limits_{\Omega }\left\vert u\right\vert ^{\xi \left( x\right)
}\left\vert \ln \left\vert u\right\vert \right\vert ^{\beta \left( x\right)
}dx\leq C_{1}\int\limits_{\Omega }\left\vert u\right\vert ^{\xi \left(
x\right) +\epsilon }dx+C_{2,}
\end{equation*}%
here $C_{1}\equiv C_{1}\left( \epsilon ,\beta^{+}\right) >0$ and $%
C_{2}\equiv C_{2}\left( \epsilon ,\beta^{+},\left\vert \Omega
\right\vert\right)>0$ are constants.\newline
Since $\xi \left( x\right) +\epsilon <\xi \left( x\right) +\beta _{1}\left(
x\right) ,$ a.e. $x\in \Omega ,$ estimating the integral on the right side
of the above equation by using Lemma \ref{var1}, we attain \eqref{4.3}.
\end{proof}

\medskip

In the following discussions, we examine elaborate properties of the
pn-spaces $S_{1,\gamma \left( x\right) ,\beta \left( x\right) ,\theta \left(
x\right) }\left( \Omega \right) $ presented in Section 2. (for other
results, see \cite{S-S-1}, \cite{S-S-2}) .

\begin{lemma}
Let $S_{1,\gamma \left( x\right) ,\beta \left( x\right) ,\theta \left(
x\right) }\left( \Omega \right) $ and $S_{1,\xi \left( x\right) ,\alpha
\left( x\right) ,\theta _{1}\left( x\right) }\left( \Omega \right) $ be the
spaces given in Definition \ref{uz}. Assume that one of the conditions given
below are satisfied

\begin{enumerate}
\item \label{varimbed}

\item[(i)] $\theta _{1}\left( x\right) \leq \theta \left( x\right) ,$ $\beta
\left( x\right) \geq \alpha \left( x\right) $ and $\xi \left( x\right) \beta
\left( x\right) =\gamma \left( x\right) \alpha \left( x\right) ,$ a.e. $x\in
\Omega ,$

\item[(ii)] $\theta _{1}\left( x\right) \leq \theta \left( x\right) ,$ $\xi
\left( x\right) \beta \left( x\right) >\gamma \left( x\right) \alpha \left(
x\right) ,$ $\gamma \left( x\right) +\beta \left( x\right) \geq \xi \left(
x\right) +\alpha \left( x\right) $ and $\beta \left( x\right) \geq \alpha
\left( x\right) +\varepsilon $ for some $\varepsilon >0$
\end{enumerate}

Under these conditions the embedding%
\begin{equation}  \label{varembed}
S_{1,\gamma \left( x\right) ,\beta \left( x\right) ,\theta \left( x\right)
}\left( \Omega \right) \subset S_{1,\xi \left( x\right) ,\alpha \left(
x\right) ,\theta _{1}\left( x\right) }\left( \Omega \right) .  \tag{4.4}
\end{equation}
holds.
\end{lemma}

\begin{proof}
First, suppose that (i) holds. Let $u\in S_{1,\gamma \left( x\right) ,\beta
\left( x\right) ,\theta \left( x\right) }\left( \Omega \right) ,$ to show
the embedding \eqref{varembed}, it is sufficient to verify the finiteness of 
$\Re ^{\xi ,\alpha ,\theta _{1}}\left( u\right) .$%
\begin{equation*}
\Re ^{\xi ,\alpha ,\theta _{1}}\left( u\right) =\int\limits_{\Omega
}\left\vert u\right\vert ^{\theta _{1}\left( x\right)
}dx+\sum_{i=1}^{n}\int\limits_{\Omega }\left\vert u\right\vert ^{\xi \left(
x\right) }\left\vert D_{i}u\right\vert ^{\alpha \left( x\right) }dx
\end{equation*}%
estimating the first integral on the right member of the above equation with
the help of Lemma \ref{var1} and second one by employing Young's inequality,
we acquire 
\begin{equation*}
\Re ^{\xi ,\alpha ,\theta _{1}}\left( u\right) \leq \left( n+1\right)
\left\vert \Omega \right\vert +\int\limits_{\Omega }\left\vert u\right\vert
^{\theta \left( x\right) }dx+\sum_{i=1}^{n}\int\limits_{\Omega }\left\vert
u\right\vert ^{\frac{\xi \left( x\right) \beta \left( x\right) }{\alpha
\left( x\right) }}\left\vert D_{i}u\right\vert ^{\beta \left( x\right) }dx.
\end{equation*}%
From the conditions, $\frac{\xi \left( x\right) \beta \left( x\right) }{%
\alpha \left( x\right) }=\gamma \left( x\right)$ that yields 
\begin{equation*}
\Re ^{\xi ,\alpha ,\theta _{1}}\left( u\right) \leq \Re ^{\gamma ,\beta
,\theta }\left( u\right) +\left( n+1\right) \left\vert \Omega \right\vert ,
\end{equation*}%
so \eqref{varembed} is gained. We note that when the case $\beta \left(
x\right) =\alpha \left( x\right) $ a.e. $x\in \Omega ,$ then $\xi \left(
x\right) =\gamma \left( x\right) $ hence \eqref{varembed} can be obtained by
similar operations as above.

Now, assume that (ii) fulfills. We need to show that $\Re ^{\xi ,\alpha
,\theta _{1}}\left( u\right) $ is finite. We have 
\begin{align*}
&\Re ^{\xi ,\alpha ,\theta _{1}}\left( u\right) =\int\limits_{\Omega
}\left\vert u\right\vert ^{\theta _{1}\left( x\right)
}dx+\sum_{i=1}^{n}\int\limits_{\Omega }\left\vert u\right\vert ^{\xi \left(
x\right) }\left\vert D_{i}u\right\vert ^{\alpha \left( x\right) }dx \\
&=\int\limits_{\Omega }\left\vert u\right\vert ^{\theta _{1}\left( x\right)
}dx+\sum_{i=1}^{n}\int\limits_{\Omega }\left\vert u\right\vert ^{\xi \left(
x\right) -\frac{\gamma \left( x\right) \alpha \left( x\right) }{\beta \left(
x\right) }}\left\vert u\right\vert ^{\frac{\gamma \left( x\right) \alpha
\left( x\right) }{\beta \left( x\right) }}\left\vert D_{i}u\right\vert
^{\alpha \left( x\right) }dx.
\end{align*}%
If we estimate the first integral on the right member of the above equation
with the help of Lemma \ref{var1} and second one by employing Young's
inequality with the exponent $\frac{\beta \left( x\right) }{\alpha \left(
x\right) }$ at every point, one can acquire that 
\begin{align*}
\Re ^{\xi ,\alpha ,\theta _{1}}\left( u\right)&\leq \int\limits_{\Omega
}\left\vert u\right\vert ^{\theta \left( x\right) }dx+\left\vert \Omega
\right\vert +\sum_{i=1}^{n}\int\limits_{\Omega }\left\vert u\right\vert
^{\gamma \left( x\right) }\left\vert D_{i}u\right\vert ^{\beta \left(
x\right) }dx \\
&+n\int\limits_{\Omega }\left\vert u\right\vert ^{\frac{\xi \left( x\right)
\beta \left( x\right) -\gamma \left( x\right) \alpha \left( x\right) }{\beta
\left( x\right) -\alpha \left( x\right) }}dx.
\end{align*}%
In the light of the condition (ii), the inequality $\frac{\xi \left(
x\right) \beta \left( x\right) -\gamma \left( x\right) \alpha \left(
x\right) }{\beta \left( x\right) -\alpha \left( x\right) }<\gamma \left(
x\right) +\beta \left( x\right) $ holds so estimating the third integral in
the right side of the last inequality by Lemma \ref{var1}, we arrive at 
\begin{align*}
\Re ^{\xi ,\alpha ,\theta _{1}}\left( u\right) &\leq \left( n+1\right)
\int\limits_{\Omega }\left\vert u\right\vert ^{\theta \left( x\right)
}dx+\left( n+1\right) \left\vert \Omega \right\vert
+\sum_{i=1}^{n}\int\limits_{\Omega }\left\vert u\right\vert ^{\gamma \left(
x\right) }\left\vert D_{i}u\right\vert ^{\beta \left( x\right) }dx \\
&\leq \left( n+1\right) \left( \Re ^{\gamma ,\beta ,\theta }\left(
u\right)+\left\vert \Omega \right\vert \right)
\end{align*}%
hence from here desired inequality is achieved. Also if $\theta _{1}\left(
x\right) =\theta \left( x\right) $ a.e. $x\in \Omega ,$ by employing the
same operations one can show \eqref{varembed}.
\end{proof}

\begin{lemma}
\label{kormet} Let $\beta ,$ $\gamma $ and $\psi $ satisfy the conditions of
Theorem \ref{bijec}, then \newline
$S_{1,\gamma \left( x\right) ,\beta \left( x\right) ,\theta \left( x\right)
}\left( \Omega \right) $ is a metric space with the metric which is defined
below. $\forall u,$ $v\in S_{1,\gamma \left( x\right) ,\beta \left( x\right)
,\theta \left( x\right) }\left( \Omega \right), $%
\begin{equation*}
d_{S_{1}}\left( u,v\right) :=\left\Vert \varphi \left( u\right) -\varphi
\left( v\right) \right\Vert _{L^{\psi \left( x\right) }\left( \Omega \right)
}+\sum_{i=1}^{n}\left\Vert \varphi _{t}^{\prime }\left( u\right)
D_{i}u-\varphi _{t}^{\prime }\left( v\right) D_{i}u\right\Vert _{L^{\beta
\left( x\right) }\left( \Omega \right) },
\end{equation*}%
here $\varphi \left( x,t\right) =\left\vert t\right\vert ^{\frac{\gamma
\left( x\right) }{\beta \left( x\right) }}t$ and for every fixed $x\in\Omega$%
, $\varphi _{t}^{\prime }\left( t\right) =\left( \frac{\gamma \left(
x\right) }{\beta \left( x\right) }+1\right) \left\vert t\right\vert ^{\frac{%
\gamma \left( x\right) }{\beta \left( x\right) }}.$
\end{lemma}

\begin{proof}
It has shown in Theorem \ref{bijec} that $\varphi \left( u\right) \in
L^{\psi \left( x\right) }\left( \Omega \right) $ \footnote{%
From now on, we denote $\varphi \left(x, u\right):=\varphi \left(
u\right)=\left\vert u\right\vert ^{\frac{\gamma \left( x\right) }{\beta
\left( x\right) }}u$ for simplicity.} and $\varphi _{t}^{\prime }\left(
u\right) D_{i}u\in L^{\beta \left( x\right) }\left( \Omega \right) $
whenever $u\in S_{1,\gamma \left( x\right) ,\beta \left( x\right) ,\theta
\left( x\right) }\left( \Omega \right) ,$ thus one can verify that $%
d_{S_{1}}\left( .,.\right) :S_{1,\gamma \left( x\right) ,\beta \left(
x\right) ,\theta \left( x\right) }\left( \Omega \right) \rightarrow \mathbb{R%
} $ satisfy the metric axioms i.e.\medskip

(i) $d_{S_{1}}\left( u,v\right) \geq 0,$ (ii) $d_{S_{1}}\left( u,v\right)
=d_{S_{1}}\left( v,u\right),$ (iii) $u=v\Rightarrow d_{S_{1}}\left(
u,v\right) =0$ \medskip

(iv) $d_{S_{1}}\left( u,v\right) =0\Rightarrow \left\Vert \varphi \left(
u\right) -\varphi \left( v\right) \right\Vert _{L^{\psi \left( x\right)
}\left( \Omega \right) }=0\Rightarrow \varphi \left( u\right) =\varphi
\left( v\right) $ since $\varphi $ is 1-1, then $u=v.$ \medskip

(v) From the subadditivity of norm, $d_{S_{1}}\left( u,v\right) \leq
d_{S_{1}}\left( u,w\right) +d_{S_{1}}\left( w,v\right) $
\end{proof}

\begin{theorem}
Under the conditions of Theorem \ref{bijec}, $\varphi $ is a homeomorphism
between the spaces $S_{1,\gamma \left( x\right) ,\beta \left( x\right)
,\theta \left( x\right) }\left( \Omega \right) $ and $L^{1,\text{ }\beta
\left( x\right) }\left( \Omega \right) \cap L^{\psi \left( x\right) }\left(
\Omega \right) .$
\end{theorem}

\begin{proof}
The function $\varphi $ is a bijection between $S_{1,\gamma \left( x\right)
,\beta \left( x\right) ,\theta \left( x\right) }\left( \Omega \right) $ and $%
L^{1,\text{ }\beta \left( x\right) }\left( \Omega \right) \cap L^{\psi
\left( x\right) }\left( \Omega \right)$ by Theorem \ref{bijec}. Thus it is
ample to prove the continuity of $\varphi $ as well as $\varphi ^{-1}$ in
the sense of topology induced by the metric $d_{S_{1}}\left( .,.\right) .$
For this, we need to show that \newline
\textit{\textbf{(i)}} $d_{S_{1}}\left( u_{m},u_{0}\right) \underset{%
m\nearrow \infty }{\longrightarrow }0\Rightarrow \varphi \left( u_{m}\right) 
{\underset{m\nearrow \infty }{\overset{L^{1,\text{ }\beta \left( x\right)
}\left( \Omega \right) \cap L^{\psi \left( x\right) }\left( \Omega \right) }{%
\longrightarrow }}\varphi \left( u_{0}\right) }$ for every \newline
$\left\{ u_{m}\right\} _{m=1}^{\infty }\in S_{1,\gamma \left( x\right)
,\beta \left( x\right) ,\theta \left( x\right) }\left( \Omega \right) $
which converges to $u_{0}$ and \medskip \newline
\textit{\textbf{(ii)}} $v_{m}{\underset{m\nearrow \infty }{\overset{L^{1,%
\text{ }\beta \left( x\right) }\left( \Omega \right) \cap L^{\psi \left(
x\right) }\left( \Omega \right) }{\longrightarrow }}}v_{0}\Rightarrow
d_{S_{1}}\left( \varphi ^{-1}\left( v_{m}\right) ,\varphi ^{-1}\left(
v_{0}\right) \right) \underset{m\nearrow \infty }{\longrightarrow }0$ for
every $\left\{ v_{m}\right\} _{m=1}^{\infty }\in L^{1,\text{ }\beta \left(
x\right) }\left( \Omega \right) \cap L^{\psi \left( x\right) }\left( \Omega
\right) $ which converges to $v_{0}.$

\noindent Since for every $v_{m}$ and $v_{0}$, there exist unique $u_{m}$
and $u_{0}\in S_{1,\gamma \left( x\right) ,\beta \left( x\right) ,\theta
\left( x\right) }\left( \Omega \right) $ such that $\varphi \left(
u_{m}\right) =v_{m}$ and $\varphi \left( u_{0}\right) =v_{0},$ the
implication \textit{(ii)} can be written equivalently, \newline
$\varphi \left( u_{m}\right) {\underset{m\nearrow \infty }{\overset{L^{1,%
\text{ }\beta \left( x\right) }\left( \Omega \right) \cap L^{\psi \left(
x\right) }\left( \Omega \right) }{\longrightarrow }}\varphi }\left(
u_{0}\right) \Rightarrow d_{S_{1}}\left( u_{m},u_{0}\right) \underset{%
m\nearrow \infty }{\longrightarrow }0$ for every $\left\{ u_{m}\right\}\in
S_{1,\gamma \left( x\right) ,\beta \left( x\right) ,\theta \left( x\right)
}\left( \Omega \right) $ which converges to $u_{0}$

\noindent Since the proofs of \textit{\textbf{(i)}} and \textit{\textbf{(ii)}%
} are similar, we only prove \textit{\textbf{(ii)}: Let }$v_{0},$ $\left\{
v_{m}\right\} _{m=1}^{\infty }\in L^{1,\text{ }\beta \left( x\right) }\left(
\Omega \right) \cap L^{\psi \left( x\right) }\left( \Omega \right) $ and $%
v_{m}\overset{L^{1,\text{ }\beta \left( x\right) }\left( \Omega \right) \cap
L^{\psi \left( x\right) }\left( \Omega \right) }{{\longrightarrow }}%
v_{0}\Leftrightarrow \varphi \left( u_{m}\right) \overset{L^{1,\text{ }\beta
\left( x\right) }\left( \Omega \right) \cap L^{\psi \left( x\right) }\left(
\Omega \right) }{{\longrightarrow }}{\varphi }\left( u_{0}\right) .$ \newline
To verify $d_{S_{1}}\left( u_{m},u_{0}\right) \rightarrow 0,$ by definition
of metric $d_{S_{1}}$ it is ample to demonstrate that 
\begin{equation*}
\left\Vert \varphi _{t}^{\prime }\left( u_{m}\right) D_{i}u_{m}-\varphi
_{t}^{\prime }\left( u_{0}\right) D_{i}u_{0}\right\Vert _{L^{\beta \left(
x\right) }\left( \Omega \right) }\rightarrow 0\text{ and }\left\Vert \varphi
\left( u_{m}\right) -\varphi \left( u_{0}\right) \right\Vert _{L^{\psi
\left( x\right) }\left( \Omega \right) }\rightarrow 0
\end{equation*}
as $m\nearrow \infty .$

\noindent Since $\varphi \left( u_{m}\right) {{\overset{L^{1,\text{ }\beta
\left( x\right) }\left( \Omega \right) \cap L^{\psi \left( x\right) }\left(
\Omega \right) }{\longrightarrow }}\varphi }\left( u_{0}\right) ,$ we have 
\newline
$\left\Vert \varphi \left( u_{m}\right) -\varphi \left( u_{0}\right)
\right\Vert _{L^{\psi \left( x\right) }\left( \Omega \right) }{\rightarrow }%
0 $ and $\left\Vert D_{i}\left( \varphi \left( u_{m}\right) -\varphi \left(
u_{0}\right) \right) \right\Vert _{L^{\beta \left( x\right) }\left( \Omega
\right) }{\rightarrow }0.$ \newline
Hence we only need to show that 
\begin{equation*}
\left\Vert \varphi _{t}^{\prime }\left( u_{m}\right) D_{i}u_{m}-\varphi
_{t}^{\prime }\left( u_{0}\right) D_{i}u_{0}\right\Vert _{L^{\beta \left(
x\right) }\left( \Omega \right) }{\longrightarrow }0,\text{ }as\text{ }%
m\nearrow \infty .
\end{equation*}%
As we know from Lemma \ref{denknorm} 
\begin{equation}  \label{4.5}
\left\Vert \varphi _{t}^{\prime }\left( u_{m}\right) D_{i}u_{m}-\varphi
_{t}^{\prime }\left( u_{0}\right) D_{i}u_{0}\right\Vert _{L^{\beta \left(
x\right) }\left( \Omega \right) }\rightarrow 0\text{ }\Leftrightarrow \sigma
_{\beta }\left( \varphi _{t}^{\prime }\left( u_{m}\right) D_{i}u_{m}-\varphi
_{t}^{\prime }\left( u_{0}\right) D_{i}u_{0}\right) \rightarrow 0  \tag{4.5}
\end{equation}%
Based on \eqref{4.5}, for $i=\overline{1,n}$ 
\begin{equation}  \label{4.6}
\sigma _{\beta }\left( \varphi _{t}^{\prime }\left( u_{m}\right)
D_{i}u_{m}-\varphi _{t}^{\prime }\left( u_{0}\right) D_{i}u_{0}\right)
=\int\limits_{\Omega }\left\vert \varphi _{t}^{\prime }\left( u_{m}\right)
D_{i}u_{m}-\varphi _{t}^{\prime }\left( u_{0}\right) D_{i}u_{0}\right\vert
^{\beta \left( x\right) }dx  \tag{4.6}
\end{equation}%
one can show that the following equality holds 
\begin{equation*}
\varphi _{t}^{\prime }\left( u_{m}\right) D_{i}u_{m}-\varphi _{t}^{\prime
}\left( u_{0}\right) D_{i}u_{0}=\left( \tfrac{\beta \left( x\right) }{\beta
\left( x\right) +\gamma \left( x\right) }\right) D_{i}\left( \varphi \left(
u_{m}\right) -\varphi \left( u_{0}\right) \right) -
\end{equation*}%
\begin{equation}  \label{4.7}
-\left( \tfrac{D_{i}\gamma .\beta -\gamma .D_{i}\beta }{\beta \left( \gamma
+\beta \right) }\right) \left( \left\vert u_{m}\right\vert ^{\frac{\gamma
\left( x\right) }{\beta \left( x\right) }}u_{m}\ln \left\vert
u_{m}\right\vert -\left\vert u_{0}\right\vert ^{\frac{\gamma \left( x\right) 
}{\beta \left( x\right) }}u_{0}\ln \left\vert u_{0}\right\vert \right) 
\tag{4.7}
\end{equation}%
Substituting \eqref{4.7} into \eqref{4.6}, we acquire 
\begin{equation*}
\sigma _{\beta }\left( \varphi _{t}^{\prime }\left( u_{m}\right)
D_{i}u_{m}-\varphi _{t}^{\prime }\left( u_{0}\right) D_{i}u_{0}\right) =
\end{equation*}%
\begin{equation*}
=\int\limits_{\Omega }\left\vert \left( \tfrac{\beta \left( x\right) }{%
\gamma \left( x\right) +\beta \left( x\right) }\right) D_{i}\left( \varphi
\left( u_{m}\right) -\varphi \left( u_{0}\right) \right) -\right.
\end{equation*}%
\begin{equation*}
\left. \text{ \ \ \ \ \ \ \ \ \ \ \ \ \ \ \ \ \ }-\left( \tfrac{D_{i}\gamma
.\beta -\gamma .D_{i}\beta }{\beta \left( \gamma +\beta \right) }\right)
\left( \left\vert u_{m}\right\vert ^{\frac{\gamma \left( x\right) }{\beta
\left( x\right) }}u_{m}\ln \left\vert u_{m}\right\vert -\left\vert
u_{0}\right\vert ^{\frac{\gamma \left( x\right) }{\beta \left( x\right) }%
}u_{0}\ln \left\vert u_{0}\right\vert \right) \right\vert ^{^{\beta \left(
x\right) }}
\end{equation*}%
taking $\beta \left( x\right) $ into the absolute value and applying known
inequality, we gain 
\begin{align*}
& \leq 2^{\beta ^{+}-1}\int\limits_{\Omega }\left\vert D_{i}\left( \varphi
\left( u_{m}\right) \right) -D_{i}\left( \varphi \left( u_{0}\right) \right)
\right\vert ^{\beta \left( x\right) }dx \\
& +C_{3}\int\limits_{\Omega }\left\vert \left\vert u_{m}\right\vert ^{\frac{%
\gamma \left( x\right) }{\beta \left( x\right) }}u_{m}\ln \left\vert
u_{m}\right\vert -\left\vert u_{0}\right\vert ^{\frac{\gamma \left( x\right) 
}{\beta \left( x\right) }}u_{0}\ln \left\vert u_{0}\right\vert \right\vert
^{\beta \left( x\right) }dx  \tag{4.8}
\end{align*}
here $C_{3}=C_{3}\left( \beta ^{+},\left\Vert \gamma \right\Vert
_{C^{1}\left( \bar{\Omega}\right) },\left\Vert \beta \right\Vert
_{C^{1}\left( \bar{\Omega}\right) }\right) >0$ is constant.

Since $\left\Vert D_{i}\left( \varphi \left( u_{m}\right) -\varphi \left(
u_{0}\right) \right) \right\Vert _{L^{\beta \left( x\right) }\left( \Omega
\right) }{\longrightarrow }0$ as $m\nearrow \infty $ the first integral in
the right member of (4.8) converges to zero when $m$ tends to
infinity.(Lemma \ref{denknorm}).

\noindent From Theorem \ref{bijec}, function $\varphi $ is bijective between
the spaces $L^{\theta \left( x\right) }\left( \Omega \right) $ and $L^{\psi
\left( x\right) }\left( \Omega \right) .$ Also since $\left\Vert \varphi
\left( u_{m}\right) -\varphi \left( u_{0}\right) \right\Vert _{L^{\psi
\left( x\right) }\left( \Omega \right) }{\longrightarrow }0,$ we arrive at%
\begin{equation}
\varphi \left( u_{m}\right) \overset{a.e}{\underset{\Omega }{{%
\longrightarrow }}}\varphi \left( u_{0}\right) \Rightarrow u_{m}\overset{a.e}%
{\underset{\Omega }{{\longrightarrow }}}u_{0}  \tag{4.9}
\end{equation}%
and,%
\begin{align*}
\sigma _{\theta }\left( u_{m}\right) &=\int\limits_{\Omega }\left\vert
u_{m}\right\vert ^{\theta \left( x\right) }dx=\int\limits_{\Omega
}\left\vert \left\vert u_{m}\right\vert ^{\frac{\gamma \left( x\right) }{%
\beta \left( x\right) }}u_{m}\right\vert ^{\psi \left( x\right) }dx \\
&=\int\limits_{\Omega }\left\vert \varphi \left( u_{m}\right) \right\vert
^{\psi \left( x\right) }dx\leq M  \tag{4.10}
\end{align*}%
for some $M>0.$

Employing (4.9), (4.10) and Vitali's Theorem \footnote{\begin{theorem}
\textbf{(Vitali)} (\cite{R})Let $(\Omega ,\Sigma ,\mu )$ be a finite measure
space, and $f_{n}:\Omega \rightarrow \mathbb{R}$ be a sequence of measurable
functions converging a.e. to a measurable $f$. Then $\left\Vert
f_{n}-f\right\Vert _{L^{1}\left( \Omega \right) }\rightarrow 0$ as $%
n\rightarrow \infty $ iff $\left\{ f_{n}:\text{ }n\geq 1\right\} $ is
uniformly integrable. When the condition is satisfied, we have%
\begin{equation*}
\lim_{n\rightarrow \infty }\int\limits_{\Omega }f_{n}d\mu
=\int\limits_{\Omega }fd\mu .
\end{equation*}%
\end{theorem}
}, we attain 
\begin{equation}
\int\limits_{\Omega }\left\vert u_{m}\right\vert ^{\theta \left( x\right)
}dx\longrightarrow \int\limits_{\Omega }\left\vert u_{0}\right\vert ^{\theta
\left( x\right) }dx,\text{ }m\nearrow \infty .  \tag{4.11}
\end{equation}%
Since $u_{m}$ converges to $u_{0}$ in measure on $\Omega ,$ using this and
(4.11), we deduce from Lemma \ref{denknorm} that 
\begin{equation}
\sigma _{\theta }\left( u_{m}-u_{0}\right) \longrightarrow 0\Rightarrow
\left\Vert u_{m}-u_{0}\right\Vert _{L^{\theta \left( x\right) }\left( \Omega
\right) }\longrightarrow 0.  \tag{4.12}
\end{equation}%
Denote $w_{m}:=$ $\left\vert u_{m}\right\vert ^{\frac{\gamma \left( x\right) 
}{\beta \left( x\right) }}u_{m}\ln \left\vert u_{m}\right\vert $ and $%
w_{0}:=\left\vert u_{0}\right\vert ^{\frac{\gamma \left( x\right) }{\beta
\left( x\right) }}u_{0}\ln \left\vert u_{0}\right\vert ,$ then%
\begin{equation*}
\sigma _{\beta }\left( w_{m}\right) =\int\limits_{\Omega }\left\vert
u_{m}\right\vert ^{\gamma \left( x\right) +\beta \left( x\right) }\left\vert
\ln \left\vert u_{m}\right\vert \right\vert ^{\beta \left( x\right) }dx
\end{equation*}%
estimating the above integral by using Lemma \ref{var3}, one can obtain 
\begin{equation*}
\sigma _{\beta }\left( w_{m}\right) \leq C_{4}\int\limits_{\Omega
}\left\vert u_{m}\right\vert ^{\theta \left( x\right) }dx+C_{5}=C_{4}\sigma
_{\theta }\left( u_{m}\right)+C_{5}.
\end{equation*}%
From (4.12), $\sigma _{\beta }\left( w_{m}\right) \leq \tilde{M}$ for all $%
m\geq 1,$ for some $\tilde{M}$ $>0$. Thus as shown above for $u_{m}$,
similarly we conclude that as $m\nearrow \infty $%
\begin{equation*}
\sigma _{\beta }\left( w_{m}-w_{0}\right) \longrightarrow 0\Rightarrow
\end{equation*}%
\begin{equation*}
\int\limits_{\Omega }\left\vert \left\vert u_{m}\right\vert ^{\frac{\gamma
\left( x\right) }{\beta \left( x\right) }}u_{m}\ln \left\vert
u_{m}\right\vert -\left\vert u_{0}\right\vert ^{\frac{\gamma \left( x\right) 
}{\beta \left( x\right) }}u_{0}\ln \left\vert u_{0}\right\vert \right\vert
^{\beta \left( x\right) }\longrightarrow 0
\end{equation*}%
hence from (4.8) we attain, 
\begin{equation*}
\left\Vert \varphi _{t}^{\prime }\left( u_{m}\right) D_{i}u_{m}-\varphi
_{t}^{\prime }\left( u_{0}\right) D_{i}u_{0}\right\Vert _{L^{\beta \left(
x\right) }\left( \Omega \right) }{\longrightarrow }0,\text{ }m\nearrow \infty
\end{equation*}%
so the proof is complete.
\end{proof}

\end{document}